\documentclass[a4paper,12pt]{article}
\usepackage{amssymb}
\usepackage{amsmath}
\usepackage{MnSymbol}
\usepackage{amscd}
\usepackage{eucal}
\usepackage{amsthm}
\usepackage{enumerate}
\usepackage{xcolor}
\usepackage{tikz-cd}
\usepackage{graphicx}
\usepackage[T1]{fontenc}

\newtheorem{theorem}{Theorem}[section]

\newtheorem{lemma}[theorem]{Lemma}
\newtheorem{proposition}[theorem]{Proposition}

\newtheorem{corollary}[theorem]{Corollary}

\newtheorem{question}[theorem]{Question}

\newtheorem{letterthm}{Theorem}
\newtheorem{lettercor}[letterthm]{Corollary}

\makeatletter
\newtheorem*{rep@theorem}{\rep@title}
\newcommand{\newreptheorem}[2]{%
\newenvironment{rep#1}[1]{%
 \def\rep@title{#2 \ref{##1}}%
 \begin{rep@theorem}}%
 {\end{rep@theorem}}}
\makeatother

\newreptheorem{theorem}{Theorem}
\newreptheorem{cor}{Corollary}

\theoremstyle{definition}
\newtheorem{definition}[theorem]{Definition}

\theoremstyle{remark}
\newtheorem{remark}[theorem]{Remark}
\newtheorem{example}[theorem]{Example}

\renewcommand{\P}{\mathcal{P}}

\newcommand{\Q}{\mathbb{Q}}
\newcommand{\R}{\mathbb{R}}
\newcommand{\V}{\mathbb{V}}
\newcommand{\Z}{\mathbb{Z}}

\newcommand{\Wh}{\mathrm{Wh}}
\newcommand{\St}{\mathrm{St}}
\newcommand{\Lk}{\mathrm{Lk}}

\DeclareMathOperator{\cd}{cd}
\DeclareMathOperator{\diam}{diam}

\DeclareMathOperator{\Hull}{Hull}

\DeclareMathOperator{\Stab}{Stab}

\newcommand{\univ}[1]{\widetilde{#1}}

\newcommand{\curlyH}{\mathcal{H}}

\newcommand{\into}{\hookrightarrow}

\usepackage{hyperref}

\title{Surface groups among cubulated hyperbolic and one-relator groups}
\author{Henry Wilton}

\newcommand{\Addresses}{{
  \bigskip
  \footnotesize

  \textsc{DPMMS, Centre for Mathematical Sciences, Wilberforce Road, Cambridge, CB3 0WB, UK}\par\nopagebreak
  \textit{E-mail address:} \texttt{h.wilton@maths.cam.ac.uk}

}}

\begin{document}

\maketitle

\begin{abstract}
Let $X$ be a non-positively curved cube complex with hyperbolic fundamental group. We prove that $\pi_1(X)$ has a non-free subgroup of infinite index unless $\pi_1(X)$ is either free or a surface group, answering questions of Gromov and Whyte (in a special case) and Wise.  A similar result for one-relator groups follows, answering a question posed by several authors.  The proof relies on a careful analysis of free and cyclic splittings of cubulated groups.
\end{abstract}

Surface groups -- the fundamental groups of closed, aspherical surfaces -- and free groups are often seen as close cousins. One way this can be made precise is to look at their subgroup structures. The Nielsen--Schreier theorem asserts that every subgroup of a free group is free, while Johansson also proved that every subgroup \emph{of infinite index} in a surface group is free \cite{johansson_topologische_1931}.  It is natural to wonder whether this property characterises free and surface groups. Conjectures and questions along these lines have been stated under various hypotheses \cite{baumslag_one-relator_2019,baumslag_open_2002,bestvina_questions_????,gardam_surface_2023,khukhro_unsolved_2023,wise_icm_2014}.

To the best of the author's knowledge, the following very general question is open.

\begin{question}\label{qu: General subgroup question}
Let $G$ be an infinite, finitely presented group, such that every subgroup of infinite index is free. Must $G$ be isomorphic to either a free group or a surface group?
\end{question}

A finitely generated, but infinitely presented, group with every proper subgroup isomorphic to $\Z$ was constructed by Ol'shanskii \cite{olshanskii_infinite_1979}. Clearly, a group $G$ with a surface subgroup can only have every subgroup of infinite index free if $G$ itself is virtually a surface group, so Question \ref{qu: General subgroup question} can be thought of as a useful precursor to the famous problem of finding surface subgroups.

We provide positive answers to Question \ref{qu: General subgroup question} for two important classes of groups. Recall that a group $G$ is \emph{cubulated} if $G=\pi_1(X)$, where $X$ is a compact, non-positively curved cube complex. Cubulated groups have been a major theme in geometric group theory for several decades, culminating in the work of Agol and Wise on the virtual Haken conjecture \cite{agol_virtual_2013,wise_structure_2021}. Our main theorem gives a strong positive answer to Question \ref{qu: General subgroup question} for cubulated hyperbolic groups.

\begin{letterthm}\label{thm: Main theorem for cubulated hyperbolic groups}
Let $G$ be a cubulated hyperbolic group. Unless $G$ is free or a surface group, $G$ has a one-ended, quasiconvex subgroup of infinite index.
\end{letterthm}

\noindent This answers questions of Gromov and Whyte in the cubulated case \cite[Questions 1.7 and 1.11]{bestvina_questions_????}. In particular, Theorem \ref{thm: Main theorem for cubulated hyperbolic groups} applies to $C'(1/6)$ small-cancellation groups \cite{wise_cubulating_2004}.

The proof makes essential use of the cubulation hypothesis, and removing it seems to be well beyond current technology. The hyperbolicity hypothesis, on the other hand, can be relaxed to the assumption that $G$ is \emph{virtually special} in the sense of Haglund--Wise \cite{haglund_special_2008}.\footnote{Agol's theorem asserts that cubulated hyperbolic groups are virtually special, so this is indeed a relaxation \cite{agol_virtual_2013}.}

\begin{lettercor}\label{cor: Corollary for virtually special groups}
Let $G$ be the fundamental group of a compact, virtually special cube complex. Unless $G$ is free or a surface group, $G$ has a one-ended, cubically convex-cocompact subgroup of infinite index.
\end{lettercor}
\begin{proof}
If $G$ is hyperbolic then the result follows from Theorem \ref{thm: Main theorem for cubulated hyperbolic groups}. Otherwise, Caprace--Haglund proved that $G$ contains a cubically convex-cocompact subgroup isomorphic to $\Z^2$ \cite[Corollary 4]{caprace_geometric_2009}, and the result follows unless $\Z^2$ has finite index. But, in the latter case, $G$ is a surface group by Bieberbach's theorem.
\end{proof}

\noindent This corollary answers a question asked by Wise in the 2014 ICM proceedings \cite[Question 13.50]{wise_icm_2014}.

\begin{remark}\label{rem: Hyperbolisation theorems}
The proof of Corollary \ref{cor: Corollary for virtually special groups}  will apply to any class of cubulated groups in which one has a `hyperbolisation' theorem, guaranteeing that a cubulated group $G$ without a $\mathbb{Z}^2$ subgroup must be hyperbolic. Such hyperbolisation theorems are also known for cube complexes without cyclic facing triples \cite{sageev_periodic_2011} and for weakly special cube complexes \cite{nakagawa_non-hyperbolic_2014}.
\end{remark}

Mapping tori of free-group automorphisms are another appealing class of groups to which Theorem \ref{thm: Main theorem for cubulated hyperbolic groups} applies.

\begin{lettercor}\label{cor: Corollary for free-by-cyclic groups}
Let $F$ be a finitely generated free group of rank at least 2.  Any semidirect product $G=F\rtimes \mathbb{Z}$ has a finitely generated, one-ended subgroup.
\end{lettercor}
\begin{proof}
Brinkmann proved that $G$ is hyperbolic unless $G$ contains a subgroup isomorphic to $\Z^2$ \cite{brinkmann_hyperbolic_2000}. Since $F$ has rank at least 2, such a subgroup must have infinite index, so the result follows unless $G$ is hyperbolic. In this case, $G$ is cubulated by the work of Hagen--Wise \cite{hagen_cubulating_2015,hagen_cubulating_2016}. Since $G$ is neither free nor a surface group, the result again follows, this time by Theorem \ref{thm: Main theorem for cubulated hyperbolic groups}.
\end{proof}

Historically, Question \ref{qu: General subgroup question} has received the most attention when $G$ is a one-relator group. In this context, the question was attributed to Fine--Rosenberger--Wienke in the Kourovka notebook \cite[Question 20.22]{khukhro_unsolved_2023}, but it has also received attention in many other works on one-relator groups, sometimes with additional hypotheses \cite{baumslag_one-relator_2019,baumslag_open_2002,ciobanu_surface_2013,fine_surface_2007}. In this context, it is sometimes referred to as the \emph{surface group conjecture}. In \cite{wilton_one-ended_2011}, the author proved the surface group conjecture for Sela's limit groups, and also for graphs of free groups with cyclic edge groups.\footnote{The latter class includes the classes of \emph{cyclically pinched} and \emph{conjugacy pinched} one-relator groups considered in \cite{ciobanu_surface_2013,fine_surface_2007}.}

Most recently, the surface group conjecture was addressed by Gardam--Kielak--Logan \cite[Conjecture 1.4]{gardam_surface_2023}, who proved it in the case when the one-relator group $G$ has two generators. Combining their work with Theorem \ref{thm: Main theorem for cubulated hyperbolic groups} and other recent advances in our understanding of one-relator groups, the surface group conjecture follows.

\begin{letterthm}\label{thm: Main theorem for one-relator groups}
Let $G$ be an infinite one-relator group. Unless $G$ is free or a surface group, $G$ has a subgroup of infinite index that is not free. 
\end{letterthm}

The alert reader will notice that the conclusion in the case of one-relator groups is slightly weaker than the conclusions for cubulated hyperbolic or free-by-cyclic groups. While Theorem \ref{thm: Main theorem for cubulated hyperbolic groups} produces a finitely generated (indeed, quasiconvex) subgroup that is not free, the non-free subgroup produced by Theorem \ref{thm: Main theorem for one-relator groups} may be infinitely generated.

\begin{remark}\label{rem: 2-generator one-relator groups}
In fact, the proof of Theorem \ref{thm: Main theorem for one-relator groups} produces a finitely generated non-free subgroup unless $G$ has two generators. In the two-generator case, the stronger result is not true: every finitely generated subgroup of infinite index in the $(1,2)$-Baumslag--Solitar group $BS(1,2)$ is either trivial or isomorphic to $\Z$, but $BS(1,2)$ contains a subgroup isomorphic to the infinitely generated non-free group $\Q_2$ of dyadic rationals. (Cf.\ Question \ref{qu: Refined one-relator statement} below.)
\end{remark}

The proof of Theorem \ref{thm: Main theorem for cubulated hyperbolic groups} is inspired by the author's proof of the corresponding statement for graphs of free groups with cyclic edge groups \cite{wilton_one-ended_2011}. That proof relied on Whitehead theory to recognise relative free splittings of free groups, so our first task is to generalise that machinery to the context of non-positively curved cube complexes.

Let $X$ be a compact, non-positively curved cube complex. We consider compact, convex subcomplexes $Y$ of the universal cover $\univ{X}$, and define a \emph{Whitehead complex} $\Wh_X(Y)$. This is a simplicial complex that encodes the intersections of the bounding hyperplanes of $Y$, and provides a topological model for the complement $\univ{X}\smallsetminus Y$. We will be especially interested in the minimal cardinality $k$ of a cut set of vertices of $\Wh_X(Y)$: if there is such a cut set, then $Y$ is called a \emph{$k$-cut}. An essential cube complex $X$ has a 0-cut if and only if $\pi_1(X)$ has more than one end.

We analyse the behaviour of the Whitehead complex when $Y$ is cut along a hyperplane. It turns out that $\Wh_X(Y)$ decomposes as a connect sum in the natural sense for simplicial complexes, in a way that generalises the work of Manning \cite{manning_virtually_2010} and Cashen--Macura \cite{cashen_line_2011} in the classical context of Whitehead graphs. Our key technical results analyse how cut sets of vertices behave under the connect sum operation.  This enables us to relate cuts to links of vertices. The following criterion, a cubical analogue of Whitehead's cut-vertex lemma, is a sample application.

\begin{letterthm}\label{thm: Generalised Whitehead's lemma}
Let $X$ be a compact, non-positively curved cube complex. If the link of every vertex is connected and has no cut simplices then $\pi_1(X)$ has at most one end.
\end{letterthm}

While 0-cuts detect free splittings, 1-cuts turn out to obstruct analysing splittings over non-trivial subgroups.  Fortunately, using recent work of Shepherd \cite{shepherd_semistability_2023}, as long as $\pi_1(X)$ is one-ended we may assume that $X$ has no 1-cuts.

Assuming $\pi_1(X)$ is one-ended, our next task is to analyse splittings over $\Z$, and to do this we restrict to the case when $\pi_1(X)$ is hyperbolic.  The approach generalises work of Cashen--Macura \cite{cashen_line_2011} and Cashen \cite{cashen_splitting_2016}, who showed that a line pattern in a free group does not split over $\Z$ if and only if there is a bound on the diameters of cut pairs in the Whitehead graph.  We prove an analogous result, Theorem \ref{thm: Slice length bound}, which asserts that $\pi_1(X)$ is cyclically indecomposable if and only if there is a uniform bound on the widths of 2-cuts. The point is that cyclic splittings correspond to \emph{periodic} 2-cuts, from which a pigeonhole argument implies a bound on the width of a 2-cut if $\pi_1(X)$ is cyclically indecomposable.  Using the projection-complex machinery of Bestvina--Bromberg--Fujiwara \cite{bestvina_constructing_2015}, and their improved version with Sisto \cite{bestvina_acylindrical_2019}, we strengthen this to a uniform bound on the widths of $k$-cuts, independent of $k$.

With these technical results in hand, we are ready to prove Theorem \ref{thm: Main theorem for cubulated hyperbolic groups}. Using the Grushko and JSJ decompositions, and since the case of graphs of free groups with cyclic edge groups is already known, we may reduce to the cyclically indecomposable case. By the above analysis, we may assume that $X$ has no 0- or 1-cuts, and that there is a uniform bound $C$ on the widths of $k$-cuts.  Agol's theorem \cite{agol_virtual_2013} implies that hyperplane stabilisers are separable, so we may pass to a finite-sheeted cover $X_0\to X$ with the property that, for some hyperplane $H$ of $\univ{X}$, every $\pi_1(X_0)$-translate of $H$ is at distance $>C$ from $H$. Cutting $X_0$ along $H$ and passing to a connected component $X'$ gives a quasiconvex subgroup that is necessarily of infinite index. But now a $0$-cut in $X'$ would lead to a $k$-cut in $X$ of width greater than $C$, which is a contradiction, so $\pi_1(X')$ is not free, as required.

The paper is structured as follows. Section \ref{sec: Preliminaries} contains preliminaries on non-positively curved complexes, cube complexes and simplicial complexes. Experts may wish to jump straight to Section \ref{sec: Free splittings of cube complexes}, where we define Whitehead complexes and use them to analyse free splittings, including proving Theorem \ref{thm: Generalised Whitehead's lemma}. In Section \ref{sec: Cyclic splittings} we extend the analysis to cyclic splittings in the hyperbolic case. Section \ref{sec: One-ended subgroups} contains the proofs of the main theorems. Section \ref{sec: Strebel} records a counterexample to a natural generalisation of Theorem \ref{thm: Main theorem for cubulated hyperbolic groups}, while some open questions are listed in Section \ref{sec: Questions}.

\subsection*{Acknowledgements}

I am grateful to Mark Hagen for useful conversations about this work, most especially for suggesting Lemma \ref{lem: Quasitree lemma}, which greatly simplifies the argument. Thanks also to Macarena Arenas, Jack Button, Chris Cashen, Mark Hagen again, Marco Linton, Lars Louder and Sam Shepherd for comments on a preliminary draft. The second draft benefitted from comments from Martin Bridson, Anthony Genevois and Chris Hruska. Special thanks to Macarena Arenas (again), Misha Kapovich and Genevieve Walsh for help with \S\ref{sec: Strebel}. Finally, thanks to the referees for suggesting many improvements.

\section{Preliminaries}\label{sec: Preliminaries}

Two kinds of cell complexes will be of particular importance in this paper. On the one hand, our main objects of study are non-positively curved cube complexes and their universal covers, which are CAT(0). On the other hand, we will study them by constructing certain simplicial complexes, called \emph{Whitehead complexes}. In this section, we recall some basic facts about all these classes of complexes. The reader is referred to the standard reference \cite{bridson_metric_1999} for background material on cube complexes and other metric polyhedral complexes, and for details of much of what follows. 

All complexes that we consider will be locally finite.

\subsection{Non-positively curved cube complxes}\label{sec: NPC cube complexes}

A \emph{cube complex} $X$ is a cell complex constructed from a disjoint union of finite-dimensional unit cubes by identifying faces isometrically \cite[Example I.7.40(4)]{bridson_metric_1999}. Further, we will also assume that cubes are embedded, and that the intersection of any two cubes is either empty or a face (i.e.\ $X$ is \emph{cubical} in the sense of Bridson--Haefliger \cite[Definition I.7.32]{bridson_metric_1999}); this can always be ensured by subdividing.

The \emph{link} $\Lk_X(x)$ of a point $x$ in a metric polyhedral complex $X$ is the space of directions at $x$, i.e.\ the space of germs of geodesics emanating from $x$ \cite[I.7.38]{bridson_metric_1999}, and admits a natural angular metric. When $X$ is a cube complex and $x$ is a vertex, $\Lk_X(x)$ has the natural structure of a simplicial complex, where the simplices are the corners of the cubes incident at $x$. 

As long as a cube complex $X$ is finite-dimensional, $X$ admits a natural complete geodesic metric in which every $k$-cube is locally isometric to $[0,1]^k\subseteq \R^k$. We will refer to this metric as the \emph{$\ell^2$-metric} on $X$. Gromov's link condition asserts that a polyhedral complex is non-positively curved -- i.e.\ the metric is locally CAT(0) -- if and only if the angular metric on the link of any point is CAT(1) \cite[4.2.A]{gromov_hyperbolic_1987}. Gromov further proved that, when $X$ is a cube complex, the link condition is satisfied if and only if the link of every vertex is \emph{flag}, meaning that, for every $k\geq 2$, every $k$-clique in the 1-skeleton of $X$ is the 1-skeleton.

The features of a non-positively curved cube complex $X$ can often be captured by studying its \emph{hyperplanes}. Let $C$ be a $k$-dimensional cube, identified with $[-1/2,1/2]^k\subseteq\R^k$. A \emph{hypercube} of $C$ is an intersection of $C$ with any coordinate hyperplane of $\R^k$. The set of hypercubes of cubes in $X$ form a natural directed system under inclusion, and gluing along these inclusions gives a cube complex with a natural map to $X$. The path components of this cube complexes are the \emph{hyperplanes} of $X$.

It is immediate from the definition that the natural map from a hyperplane $H$ to $X$ is a local isometry. Therefore, if $X$ is non-positively curved then $H$ is too.

Since hyperplanes are constructed by gluing together hypercubes of $X$, each hyperplane $H\to X$ admits a natural \emph{(closed) cubical neighbourhood} $N(H)$, which is a bundle over $H$ with fibre a closed interval. We may write $N_X(H)$ when we want to emphasise that $H$ is a hyperplane of $X$. Likewise, $N(H)$ has a natural \emph{interior} $\mathring{N}(H)$, also a bundle over $H$ with fibre an open interval.  The hyperplane $H$ is said to be \emph{two-sided} when these bundles are trivial.

If the map $H\to X$ is injective we say that $H$ is \emph{embedded}. In this case, the interior $\mathring{N}(H)$ is also naturally embedded in $X$. The complement $X_0=X\smallsetminus \mathring{N}(H)$ is a locally convex subcomplex of $X$, and is said to be obtained by \emph{cutting $X$ along $H$}. In this case, we write $X=X_0\cup_H$. When $H$ is two-sided, this expresses $X$ as a \emph{graph of spaces} in the sense of Scott and Wall \cite{scott_topological_1979}. As long as $X$ itself is connected, the vertex space $X_0$ has at most two components, and if there are two then we write $X=X_1\cup_H X_2$ in the obvious way. 

Finally, we can make use of the following important lemma \cite[Proposition II.4.14]{bridson_metric_1999}.

\begin{lemma}\label{lem: Local isometries are pi1 injective}
Let $X,Y$ be complete, connected, geodesic, metric spaces. If $X$ is non-positively curved then any local isometry $Y\to X$ is injective on fundamental groups.
\end{lemma}

\noindent Thus, the inclusion of a hyperplane is $\pi_1$-injective, so cutting along $H$ also decomposes the fundamental group $G=\pi_1(X)$ as a graph of groups in the sense of Serre \cite{serre_trees_2003}, with a single edge group $\pi_1(H)$.

\subsection{CAT(0) cube complexes}\label{sec: CAT(0) cube complexes}

When $X$ is non-positively curved, the natural geodesic metric on the universal cover $\univ{X}$ is CAT(0), by the Cartan--Hadamard theorem \cite[\S4]{gromov_hyperbolic_1987}. This gives us access to the standard tools of CAT(0) geometry detailed in \cite{bridson_metric_1999}, which are especially useful when studying a closed, convex subspace $Y\subseteq \univ{X}$. For instance, we may define the orthogonal projection $\pi_Y:\univ{X}\to Y$ as in \cite[Proposition II.2.4]{bridson_metric_1999}.

One important consequence of the uniqueness of geodesics in CAT(0) spaces is that any local isometry between CAT(0) spaces is a global isometric embedding \cite[Proposition II.4.14]{bridson_metric_1999}. Thus, the map from a hyperplane  $H\to X$ lifts to a convex embedding of the universal cover $\univ{H}\into \univ{X}$, and so the hyperplanes of $\univ{X}$ are the universal covers of the hyperplanes of $X$.

The natural map $N(H)\to X$ is also a local isometry. Thus, the map $N(H)\to X$ also lifts to an embedding
\[
\widetilde{N(H)}=N(\univ{H})\into\univ{X}\,.
\]
The interval bundle $N(\univ{H})$ is necessarily trivial because $\univ{H}$ is simply connected, so $\univ{H}$ is embedded and two-sided, and cutting along $\univ{H}$ realises $\univ{X}$ as a graph of spaces. Because $\univ{X}$ is simply connected it follows that $\univ{H}$ separates it into two, and so we have 
\[
\univ{X}=\univ{U}_+\cup_{\univ{H}}\univ{U}_-
\]
for two convex subcomplexes $\univ{U}_\pm$. These subcomplexes are called the \emph{half-spaces} of $\univ{H}$.

It is usually important when studying CAT(0) cube complexes to rule out `uninteresting' half-spaces that are within a bounded distance of the corresponding hyperplane.

\begin{definition}\label{def: Essential cube complex}
A hyperplane $\univ{H}$ of a CAT(0) cube complex $\univ{X}$ is called \emph{inessential} if some half-space of $\univ{H}$ is within a bounded neighbourhood of $\univ{H}$. If $\univ{X}$ has no inessential hyperplanes then $\univ{X}$ itself is called \emph{essential}. We will call a non-positively curved complex $X$ essential or inessential according to its universal cover.
\end{definition}

Fortunately, we may always assume that $X$ is essential by invoking a result of Caprace--Sageev \cite[Proposition 3.5]{caprace_rank_2011}.

\begin{lemma}\label{lem: Essential complexes}
Every compact, non-positively curved cube complex $X$ deformation retracts to an essential, convex subcomplex of its first subdivision.
\end{lemma}

Hyperplanes also make it possible to define a notion of convex hull for a subset of $\univ{X}$.

\begin{definition}\label{def: Hull}
For any subset $Y$ of a CAT(0) cube complex $\univ{X}$, the \emph{cubical convex hull} $\Hull(Y)$ is defined to be the intersection of the half-spaces of $\univ{X}$ that contain $Y$.
\end{definition}

It is important to be able to control the geometry of convex hulls, and we can do this using a fundamental theorem of Haglund.  This theorem makes use of the \emph{$\ell^1$-metric} on (the vertices of) $\univ{X}$.

\begin{definition}\label{def: l1 metric}
Let $\univ{X}$ be a CAT(0) cube complex. The $\ell^1$-metric on the 1-skeleton $X^{(1)}$ declares the distance between two points $x$ and $y$ to be equal to the length of the shortest path in $X^{(1)}$ between them.
\end{definition}

Note that the $\ell^1$-metric on $X^{(1)}$ is also geodesic, and that the $\ell^1$ distance between two vertices is exactly the number of hyperplanes that separate them \cite[Lemma 3.8]{sageev_ends_1995}.

The following lemma proves that orthogonal projection from a convex subcomplex to a hyperplane must land in their intersection, if it is non-empty.

\begin{lemma}\label{lem: Projecting subcomplexes onto hyperplanes}
Let $\univ{X}$ be a CAT(0) cube complex, let $Y\subseteq \univ{X}$ be a convex subcomplex, and let $H$ be a hyperplane of $\univ{X}$. If $H\cap Y$ is non-empty then $\pi_H(y)\in Y$, for all $y\in Y$.
\end{lemma}
\begin{proof}
Let $y\in Y$ and $x=\pi_{H\cap Y}(y)$; we shall prove that $x=\pi_H(y)$. Since the result is trivial if $y=x$, we may assume that $y\neq x$ and choose $z\neq x$ in the intersection of the geodesic $[x,y]$ with the carrier $N_Y(H\cap Y)$.  But $x=\pi_{H\cap Y}(z)$ so, in the product coordinates on $N_Y(H\cap Y)\cong(H\cap Y)\times [-1/2,1/2]$, we have $z=(x,t)$ for some $t$. The product coordinates extend naturally to $N_{\univ{X}}(H)$, so the geodesic $[x,z]$ is orthogonal to every geodesic in $H$ through $x$, and hence $[x,y]$ is also. Therefore, $x=\pi_H(y)$ and the result follows.
\end{proof}

The next lemma is useful for relating the $\ell^1$ and $\ell^2$ distances.

\begin{lemma}\label{lem: l2 geodesics and separating hyperplanes}
Let $Y_1,Y_2$ be  convex subcomplexes of a CAT(0) cube complex $\univ{X}$. Suppose that a geodesic $\gamma$ minimises the $\ell^2$-distance from $Y_1$ to $Y_2$. A hyperplane $H$ separates $Y_1$ from $Y_2$ if and only if $\gamma$ crosses $H$.
\end{lemma}
\begin{proof}
The necessity is obvious, so it remains to prove that, if $\gamma$ crosses $H$, then $H$ separates $Y_1$ from $Y_2$.  Suppose that $\gamma(0)\in Y_1$ and $\gamma$ crosses $H$ at $\gamma(t)$. Because $\gamma$ minimises the distance, it makes an angle of at least $\pi/2$ with each geodesic in $Y_1$ through $\gamma(0)$. By Lemma \ref{lem: Projecting subcomplexes onto hyperplanes}, $\pi_H(\gamma(0))\in Y_1$. Hence, if $H$ intersects $Y_1$ then $\gamma(0)$, $\gamma(t)$ and $\pi_H(\gamma(0))$ form a triangle with two right angles, contradicting the hypothesis that $\univ{X}$ is CAT(0). Therefore, $H$ is disjoint from $Y_1$, and likewise from $Y_2$. Since $\gamma$ crosses $H$ once, it follows that $H$ separates $Y_1$ from $Y_2$, as required.
\end{proof}

We can now state Haglund's theorem \cite[Theorem H]{haglund_finite_2008}

\begin{theorem}[Haglund]\label{thm: Haglund's theorem}
Let $\univ{X}$ be a $d$-dimensional CAT(0) cube complex and $Y\subseteq \univ{X}^{(0)}$ a set of vertices that is $\kappa$-quasiconvex in the $\ell^1$-metric. There is $R=R(d,\kappa)$ such that every point of $\Hull(Y)$ is at distance at most $R$ from a point of $Y$.
\end{theorem} 

One useful consequence is to characterise geometrically well-behaved subgroups.  A subgroup $H$ of the fundamental group of a non-positively curved cube complex $X$ is called \emph{cubically convex-cocompact} if $H$ acts cocompactly on a subcomplex $Y\subseteq\univ{X}$ that is convex in the $\ell^2$-metric. Note that, by \cite[Theorem 2.13]{haglund_finite_2008}, a subcomplex is convex in the $\ell^2$-metric if and only if it is convex in the $\ell^1$-metric.
\begin{corollary}\label{cor: CCC vs l1 quasiconvex}
Let $X$ be a compact, non-positively curved cube complex. A subgroup $H$ of $\pi_1(X)$ is cubically convex-cocompact if and only if the $H$-orbit of some vertex is quasiconvex in the $\ell^1$-metric.
\end{corollary}

\subsection{Simplicial complexes}

As the above discussion makes clear, our investigations of cube complexes will often be metric. By contrast, when simplicial complexes arise, we will be interested in their combinatorial structure, and especially interested in collections of vertices and simplices that separate. Here, we introduce the necessary terminology.

We already saw that, in a metric complex, the link of a point can be thought of metrically as a space of directions. In a simplicial complex, it is convenient to give an equivalent combinatorial definition, which also extends to certain subcomplexes. A subcomplex $B$ of a simplicial complex $A$ is called \emph{full} if, whenever all the vertices of a simplex $\sigma$ are in $B$, then $\sigma$ itself is in $B$.

\begin{figure}[htp]
\begin{center}
\begin{tikzpicture}[scale = .60]
\fill[gray!50] (0-6,0) -- (4.5-6,0) -- (0-6,4.5);
\fill[gray!50] (0-6,0) -- (0-6,4.5) -- (-4.5-6,0);
\fill[gray!50] (0-6,0) -- (-4.5-6,0) -- (0-6,-4.5);
\fill[gray!50] (0-6,0) -- (0-6,-4.5) -- (4.5-6,0);
\draw[thick] (-4.5-6,0) -- (4.5-6,0);
\draw[thick] (0-6,-4.5) -- (0-6,4.5);
\draw[thick, white] (4.5-6,0) -- (0-6,4.5);
\draw[thick, white] (0-6,4.5) -- (-4.5-6,0);
\draw[thick, white] (-4.5-6,0) -- (0-6,-4.5);
\draw[thick, white] (0-6,-4.5) -- (4.5-6,0);

\node at (-5.5,0.5) {$B$};
\node at (-6,-5.25) {$\mathrm{St_A}(B)$};

\draw[thick, gray!50] (-4.5+6,0) -- (4.5+6,0);
\draw[thick, gray!50] (0+6,-4.5) -- (0+6,4.5);
\draw[very thick] (4.5+6,0) -- (0+6,4.5);
\draw[very thick] (0+6,4.5) -- (-4.5+6,0);
\draw[very thick] (-4.5+6,0) -- (0+6,-4.5);
\draw[very thick] (0+6,-4.5) -- (4.5+6,0);

\node at (6.5,0.5) {$B$};
\node at (6,-5.25) {$\mathrm{Lk_A}(B)$};
\end{tikzpicture}
\end{center}
\caption{\label{fig: Simplicial stars and links}  An example of an open star and a link. The example is a closed disc, triangulated as the union of four 2-simplices; the subcomplex $B$ is the vertex in the centre. On the left, the open star is illustrated as the union of the grey open 2-simplices and the black 0- and 1-simplices. On the right, the link is illustrated: it is the union of the 0- and 1-simplices shaded black.}
\end{figure}

\begin{definition}[Stars and links]\label{def: Stars and links}
Let $A$ be a simplicial complex and $B\subseteq A$ a full subcomplex. The \emph{open star} of $B$, $\St_A(B)$, is the union of the interiors of the simplices of $A$ that contain a simplex of $B$ as a face. The \emph{link} of $B$, $\Lk_A(B)$, is the set of faces of simplices of $\St_A(B)$ that are disjoint from $B$. See Figure \ref{fig: Simplicial stars and links}.
\end{definition}

Note that $\St_A(B)$ may not be a subcomplex of $A$, but $\Lk_A(B)$ always is. Furthermore, after putting a suitable metric on $A$, for any vertex $v$, $\Lk_A(v)$ coincides with the metric definition of link given in \S\ref{sec: NPC cube complexes}.

As mentioned above, we are especially interested in the separation properties of subcomplexes.

\begin{definition}[Cut sets and connectivity]\label{def: Cut sets and connectivity}
Let $A$ be a simplicial complex and $B$ a full subcomplex. Write $A_B:=A\smallsetminus \St(B)$. The subcomplex $B$ is said to \emph{separate} $A$ if $A_B$ is disconnected. If $B$ is a simplex we say $B$ is a \emph{cut simplex}.

In the specific case where $B$ is a set of $k$ pairwise non-adjacent vertices, we say that $B$ is a \emph{cut set} for $A$. If $\#B=1$ we say $B$ is a \emph{cut vertex}, and if $\#B=2$ we say $B$ is a \emph{cut pair}.  If $A$ has no cut set of cardinality at most $k$, then $A$ is said to be \emph{$k$-vertex-connected}.
\end{definition}

\section{Free splittings}\label{sec: Free splittings of cube complexes}

In 1936, J.H.C.\ Whitehead explained how to determine whether an element $w$ of a free group $F$ is primitive. For a fixed basis $B$ of $F$, his algorithm uses a graph $\Wh_B(w)$ which is now known as the \emph{Whitehead graph} \cite{whitehead_equivalent_1936}. His algorithm can equally be interpreted as searching for free splittings $F\cong\langle w\rangle* F'$. Many elaborations and generalisations of Whitehead's argument have been given over the years by, among others, Lyon \cite{lyon_incompressible_1980}, Stallings \cite{stallings_whitehead_1999} and Stong \cite{stong_diskbusting_1997}.  Diao--Feighn, building on work of Gersten \cite{gersten_whitehead_1984}, gave an algorithm to compute the Grushko decomposition of a graph of free groups \cite{diao_grushko_2005}.\footnote{See the recent preprint of Dicks for a comprehensive history of this topic \cite[\S2]{dicks_whitehead_2017}.}

The power of this point of view is that  \emph{global} information about a group $G$ -- whether or not $G$ is free, or splits freely, or most generally the structure of the Grushko decomposition of $G$ --  is encoded in the \emph{local} information of a Whitehead graph, which can be read off a presentation. The following theorem, which is the result of combining Whitehead's algorithm with a theorem of Shenitzer \cite[Theorem 5]{shenitzer_decomposition_1955}, exemplifies this idea in the case of one-relator groups.

\begin{theorem}\label{thm: Shenitzer's theorem}
Consider a one-relator group $G=F/\llangle w\rrangle$ and fix a basis $B$ for $F$. If the Whitehead graph $\Wh_B(w)$ is connected without cut vertices, then $G$ does not split freely.
\end{theorem}

\noindent Note that $\Wh_B(w)$ is exactly the link of the unique vertex in the natural presentation complex for $G$. Our goal is to develop a similar theory for non-positively curved cube complexes. To this end, our first task is to generalise the notion of Whitehead graph to the context of non-positively curved cube complexes.

\subsection{Bounding hyperplanes and Whitehead complexes}

In this section, we introduce the \emph{Whitehead complexes} that will be our main technical tool. While the use we put them to here appears to be new, they bear some similarities with \emph{crossing graphs} and \emph{crossing complexes} for CAT(0) cube complexes, which have been studied extensively in a wide variety of contexts -- see, for instance, \cite{genevois_morphisms_2022,hagen_weak_2014,niblo_singularity_2002,rowlands_relating_2023}, among many other examples.

Let $X$ be a non-positively curved cube complex, let $\univ{X}\to X$ be its universal cover and let $Y\subseteq \univ{X}$ be a convex subcomplex. Equivalently, we may think of $Y$ as a CAT(0) complex equipped with a locally isometric combinatorial map $Y\to X$ -- any such immersion lifts to an embedding $Y\into \univ{X}$, and all such embeddings differ by a deck transformation of $X$.

\begin{definition}\label{def: Bounding hyperplane}
A hyperplane $H$ of $\univ{X}$ \emph{bounds} $Y$ if it is disjoint from $Y$ but no other hyperplane separates $Y$ from $H$. Equivalently, $H$ and $Y$ are disjoint, but $Y$ intersects the (closed) regular neighbourhood $N(H)\cong H\times[0,1]$. We may orient the interval factor so that the intersection is of the form $H_Y\times \{0\}$, where $H_Y$ is a convex subcomplex of $H$ called the \emph{bounding hyperplane component of $Y$ in $H$}.
\end{definition}

The Whitehead complex associated to such a $Y$ encodes the way that the bounding hyperplanes of $Y$ intersect each other.

\begin{definition}\label{def: Whitehead complex}
Let $X$ be a non-positively curved cube complex, and $Y$ a convex subcomplex of the universal cover $\univ{X}$. The \emph{Whitehead complex} of $Y$, $\Wh_X(Y)$, is a simplicial complex defined as follows. The vertices of $\Wh_X(Y)$ are the bounding hyperplanes of $Y$. A finite set of bounding hyperplanes $\curlyH=\{H_0,\ldots, H_k\}$ spans a $k$-simplex $\sigma_\curlyH$ in $\Wh_X(Y)$ if and only if the intersection $\bigcap_{i=0}^kH_i$ is non-empty.
\end{definition}

\begin{remark}\label{rem: WCs are flag}
The Helly property for hyperplanes of CAT(0) cube complexes, proved by Sageev, asserts that if a set of hyperplanes intersect pairwise then their intersection is non-empty \cite[Theorem 4.14]{sageev_ends_1995}. Hence, Whitehead complexes are flag.
\end{remark}

The name  is justified by the following example.

\begin{example}\label{eg: Whitehead's graph}
Let $F$ be a free group, thought of as the fundamental group of a bouquet of circles $R$. (This is equivalent to choosing a basis $B$ for $F$.) The conjugacy class of a non-trivial element $w\in F$ can be represented by an immersion $w:S^1\to R$. The mapping cylinder $X$ of this immersion is naturally a non-positively curved square complex. Let $*$ be the one-point space and consider the map $*\to X$ with image equal to the unique vertex of $R$. Then $\Wh_X(*)$ is the Whitehead graph $\Wh_B(w)$, subdivided once.

More generally, in \cite{cashen_line_2011}, Cashen and Macura defined \emph{generalised Whitehead graphs} $\Wh_B(T)\{w\}$ associated to a word $w$ in a free group $F$, for any subtree $T$ of the Cayley tree $\univ{R}$ of $F$. Our complex $\Wh_X(T)$ is the graph obtained from $\Wh_B(T)\{w\}$ by subdividing once.
\end{example}

The case when $Y$ is a vertex provides another important example.

\begin{example}\label{eg: Whitehead graphs of vertices are links}
Consider $\tilde{y}$ a vertex of $\univ{X}$ with image $y\in X$. The bounding hyperplanes of $\tilde{y}$ are dual to the edges incident at $y$, and a set of such hyperplanes cross exactly when the edges bound a cube with corner $\tilde{y}$. Therefore, $\Wh_X(\tilde{y})\cong\Lk_X(y)$ as simplicial complexes.
\end{example}

Whitehead complexes are useful because they provide a topological model for the complement of $Y$: we shall see that $\Wh_X(Y)$ is homotopy equivalent to the complement $\univ{X}\smallsetminus Y$. In order to describe this homotopy equivalence, it is useful to define the \emph{link} of $Y$, and to extend our discussion of bounding hyperplanes to \emph{codimension-$k$ hyperplanes}.

\begin{definition}\label{def: Codimension-k bounding hyperplanes}
Consider a set of $k$ distinct hyperplanes $H_1,\ldots,H_k$ of $Y$ that pairwise intersect. The intersection
\[
K=\bigcap_{i=1}^k H_i
\]
is called a \emph{codimension-$k$ hyperplane}. It has a closed regular neighbourhood $N(K)\cong K\times [0,1]^k$. If the $H_i$ are all bounding hyperplanes of $Y$ then $K$ is called a \emph{bounding codimension-$k$ hyperplane of $Y$}. In this case, we may orient the interval factors so that
\[
N(K)\cap Y=K_Y\times\{0\}^k
\]
for some convex subcomplex $K_Y$, called the \emph{bounding codimension-$k$ hyperplane component of $K$ in $Y$}.
\end{definition}

The \emph{link} of $K_Y\times\{0\}^k$ in $K_Y\times [0,1]^k$ consists of unit vectors orthogonal to $K_Y\times\{0\}^k$, and is naturally of the form $K_Y\times \sigma_K$, where $\sigma_K$ is a $(k-1)$-simplex. 

Whenever a codimension-$k$ hyperplane $K$ is contained in a codimension-$l$ hyperplane $L$ (necessarily with $l\geq k$), there are natural inclusions $K_Y\subseteq L_Y$ and $\sigma_L\subseteq \sigma_K$, which in turn lead to inclusions
\begin{equation}\label{eqn: Link inclusions}
K_Y\times \sigma_L\subseteq  K_Y\times\sigma_K\,,\,K_Y\times \sigma_L\subseteq  L_Y\times\sigma_L\,.
\end{equation}
Gluing along these inclusions defines the \emph{link} of the subcomplex $Y$.

\begin{definition}\label{defn: Link of a subcomplex}
The \emph{link} of a convex subcomplex $Y$ of a CAT(0) cube complex $\univ{X}$ is the cell complex $\Lk_{\univ{X}}(Y)$ defined from the disjoint union
\[
\coprod_K K_Y\times\sigma_K
\]
over all bounding hyperplane components of any codimension, by gluing along the natural inclusions (\ref{eqn: Link inclusions}).  As usual, when $\univ{X}$ is the universal cover of a non-positively curved cube complex $X$, we will write $\Lk_X(Y)$.
\end{definition}

\begin{remark}\label{def: Metric description of link}
The link $\Lk_{\univ{X}}(Y)$  can also be described metrically as the germs of geodesics in $\univ{X}$ that are orthogonal to $Y$. Thus, orthogonal projection defines a natural map $\hat{\pi}_Y:\univ{X}\smallsetminus Y\to \Lk_{\univ{X}}(Y)$.
\end{remark}

We can now prove that $\Wh_X(Y)$ is homotopy equivalent to the complement of $Y$ in two stages.

\begin{lemma}\label{lem: Complements and links}
For any convex subcomplex $Y$ of a CAT(0) cube complex $\univ{X}$, the orthogonal projection $\hat{\pi}_Y:\univ{X}\smallsetminus Y\to \Lk_{\univ{X}}(Y)$ is a homotopy equivalence.
\end{lemma} 
\begin{proof}
As in the remark, we identify the link $\Lk_{\univ{X}}(Y)$ with the set of germs of geodesics orthogonal to $Y$. Since $Y$ is a subcomplex, every such germ extends uniquely to a geodesic of length $1$. Therefore, $\hat{\pi}_Y$ has a well-defined right-inverse $\rho:\Lk_{\univ{X}}(Y)\to \univ{X}\smallsetminus Y$ sending the germ of $\gamma$ to $\gamma(1)$. For any $x\notin Y$, let $\gamma_x$ denote the unique geodesic from $\hat{\pi}_Y(x)$ to $x$. Then
\[
(x,t)\mapsto \gamma_x((1-t)d(x,Y)+t)
\]
defines a homotopy equivalence between the identity and $\rho\circ\hat{\pi}_Y$, which completes the proof.
\end{proof}

Finally, we relate links to Whitehead complexes by contracting bounding hyperplane components.

\begin{lemma}\label{lem: Links and Whitehead graphs}
For any convex subcomplex $Y$ of a CAT(0) cube complex $\univ{X}$, there is a homotopy equivalence $\Lk_{\univ{X}}(Y)\simeq \Wh_{\univ{X}}(Y)$.
\end{lemma}
\begin{proof}
The link $\Lk_{\univ{X}}(Y)$ is homeomorphic to the realisation of a complex of spaces in the sense of Hatcher, with underlying complex the barycentric subdivision of $\Wh_{\univ{X}}(Y)$. Each vertex corresponds to a codimension-$k$ hyperplane $K$, the associated space is the hyperplane component $K_Y$, and the attaching maps are given by the natural inclusions. Since each $K_Y$ is contractible, it follows that $\Lk_{\univ{X}}(Y)\simeq \Wh_{\univ{X}}(Y)$ by \cite[Proposition 4G.1]{hatcher_algebraic_2002}.
\end{proof}

The previous two lemmas combine to give the following fundamental result.

\begin{lemma}\label{lem: Topological type of Whitehead complex}
For any convex subcomplex $Y$ of a CAT(0) cube complex $\univ{X}$, there is a homotopy equivalence $\univ{X}\smallsetminus Y\simeq\Wh_{\univ{X}}(Y)$.
\end{lemma}
\noindent The lemma can also be deduced from the nerve theorem, by noting that $\Wh_{\univ{X}}(Y)$ is the nerve of a natural open cover of $\univ{X}\smallsetminus Y$.

Finally, it will be useful to understand the links of vertices in Whitehead complexes. (Note that here,  in the setting of simplicial complexes, we mean link in the sense of Definition \ref{def: Stars and links}.)

\begin{lemma}\label{lem: Links in Whitehead complexes}
Let $Y$ be a convex subcomplex of a CAT(0) cube complex $\univ{X}$ and let $H$ be a bounding hyperplane of $Y$. Then
\[
\Lk_{\Wh_{\univ{X}}(Y)}(H)\cong\Wh_H(H_Y)\,.
\]
\end{lemma}
\begin{proof}
Since both simplicial complexes are flag, it suffices to define an isomorphism at the level of 1-skeleta.

The vertices of $\Lk_{\Wh_{\univ{X}}(Y)}(H)$ are, by definition, the bounding hyperplanes of $Y$ that cross $H$. For every such hyperplane $K$, the intersection $H\cap K$ is a bounding hyperplane of $H_Y$ in $H$. Conversely, every bounding hyperplane of $H_Y$ in $H$ extends uniquely to a bounding hyperplane of $Y$ that crosses $H$. Thus, $K\mapsto H\cap K$ defines the required bijection on the 0-skeleta.  The bijection extends to an isomorphism of 1-skeleta, because bounding codimension-three hyperplanes of $Y$ contained in $H$ (which correspond to edges of $\Lk_{\Wh_{\univ{X}}(Y)}(H)$) are the same thing as bounding codimension-2 hyperplanes of $H_Y$ in $H$ (which correspond to edges of $\Wh_H(H_Y)$).
\end{proof}

\subsection{Ends and splittings}

Invoking a famous theorem of Stallings, the number of \emph{(Freudenthal) ends} of the universal cover of a non-positively curved cube complex $X$ detects the free splittings of the fundamental group \cite{stallings_torsion-free_1968}. (Note that $\pi_1(X)$ is necessarily torsion-free.)

\begin{theorem}[Stallings]\label{thm: Ends and splittings}
Let $X$ be a compact, non-positively curved cube complex.
\begin{enumerate}[(i)]
\item $\mathrm{Ends}(\univ{X})=0$ if and only if $\pi_1(X)$ is trivial.
\item $\mathrm{Ends}(\univ{X})>1$ if and only if $\pi_1(X)$ splits freely.\footnote{A non-trivial free splitting corresponds to an action on a simplicial tree without a global fixed point and with trivial edge stabilisers. Note that $\mathbb{Z}\cong 1*_1$ splits freely.}
\end{enumerate}
\end{theorem}

We will see that the number of ends can be conveniently encoded in the connectivity properties of Whitehead graphs. The next result is prototypical, and deals with the case when $X$ is simply connected.

\begin{proposition}\label{prop: Whitehead complex and 0 ends}
Let $X$ be a compact, non-positively curved cube complex. There is a compact CAT(0) cube complex $Y$ mapping to $X$ with $\Wh_X(Y)=\varnothing$ if and only if $\pi_1(X)$ is trivial.
\end{proposition}
\begin{proof}
The Whitehead graph $\Wh_X(Y)$ is empty if and only if $Y=\univ{X}$. Thus such a compact $Y$ exists if and only if $\univ{X}$ is compact, meaning that $\pi_1(X)$ is finite, i.e.\ trivial.
\end{proof}

In  \S\ref{sec: CAT(0) cube complexes}, we saw that is is often convenient to assume that our cube complexes are essential.  In fact, when studying free splittings, we will only need a weaker notion.

\begin{definition}\label{def: Freely essential cube complex}
A CAT(0) cube complex $\univ{X}$ is called \emph{freely inessential} if some (necessarily compact) hyperplane $H \subseteq\univ{X}$ has the property that one of the associated half-spaces is compact. Otherwise, $\univ{X}$ is called \emph{freely essential}. 
\end{definition}

\begin{remark}\label{rem: Freely essential remarks}
Any compact CAT(0) cube complex is freely inessential, unless it is a point, and any essential cube complex is freely essential.
\end{remark}

The next result explains how Whitehead complexes encode free splittings of the fundamental group.

\begin{proposition}\label{prop: Detecting free splittings using Whitehead graphs}
Let $\univ{X}$ be a freely essential CAT(0) cube complex. There is a compact, convex $Y\subseteq\univ{X}$ such that $\Wh_{\univ{X}}(Y)$ is non-empty and disconnected if and only if $\univ{X}$ has more than one end.
\end{proposition}
\begin{proof}
Suppose such a $Y$ exists. Since $\Wh_{\univ{X}}(Y)$ is disconnected, Lemma \ref{lem: Topological type of Whitehead complex} implies that $\univ{X}\smallsetminus Y$ is also disconnected, so $Y$ separates $\univ{X}$ into at least two components. If one of these components is bounded then any bounding hyperplane in the corresponding component of $\Wh_{\univ{X}}(Y)$ has a compact half space, so $\univ{X}$ is freely inessential. Otherwise, all components are unbounded so $\univ{X}$ has more than one end.

For the converse direction, suppose that $\univ{X}$ has more than one end, so there is a compact subcomplex $Y\subseteq \univ{X}$ such that $\univ{X}\smallsetminus Y$ has at least two unbounded components. By Theorem \ref{thm: Haglund's theorem}, we may replace $Y$ by its convex hull and assume that $Y$ is convex. Since $\univ{X}\smallsetminus Y$ has at least two components, so does $\Wh_{\univ{X}}(Y)$ by Lemma \ref{lem: Topological type of Whitehead complex}.
\end{proof}

The proposition connects Whitehead complexes and free splittings, but the resulting criterion is ineffective because there is no bound on the possible size of a $Y$ for which $\Wh_X(Y)$ is disconnected. To make the criterion effective, we need to understand how to build up Whitehead complexes inductively by gluing, generalising the \emph{splicing} construction of Manning \cite{manning_virtually_2010} used by Cashen--Macura \cite{cashen_line_2011}.

\subsection{Connected sums of simplicial complexes}

The definition of the connected sum of two simplicial complexes is directly analogous to the definition for manifolds.

\begin{definition}[Connected sum of complexes]\label{def: Splicing simplicial complexes}
Let $A,B$ be simplicial complexes, let $a\in A$, $b\in B$ be vertices, and let $\phi:\Lk_A(a)\to\Lk_B(b)$ be an isomorphism. The complex
\[
A\#_\phi B:=A_a\cup_\phi B_b
\]
is the \emph{connected sum of $A$ to $B$ along $\phi$}. (Often, we will have a natural identification $v=a=b$, and $\phi$ will be implicit. In this case, we may write $A\#_vB$ instead.)
\end{definition} 

\begin{remark}\label{rem: Splicing and connected sum}
Definition \ref{def: Splicing simplicial complexes} generalises the \emph{splicing} construction used by Manning to analyse Whitehead graphs in finite-index subgroups \cite{manning_virtually_2010}, and exploited further by Cashen and Macura \cite{cashen_splitting_2016,cashen_line_2011}.  Up to taking first subdivisions, Manning's definition coincides with Definition \ref{def: Splicing simplicial complexes} when restricted to graphs. (Recall from Example \ref{eg: Whitehead's graph} that the classical Whitehead graph is the first subdivision of our Whitehead complex.)
\end{remark}

The next result explains how cutting $Y$ along a hyperplane decomposes the Whitehead complex as a connected sum. When a convex subcomplex $Y\subseteq \univ{X}$ intersects a hyperplane $H$ of $\univ{X}$, we will abuse notation by identifying $H$ with $H\cap Y$. Cutting $Y$ along $H$ then leads to the decomposition
\[
Y=Y_1\cup_H Y_2\,;
\]
note that $H$ is then a bounding hyperplane of both $Y_1$ and $Y_2$.

\begin{proposition}[Cutting and connected sum]\label{prop: Constructing Whitehead complexes by splicing}
Let $Y$ be a convex subcomplex of a CAT(0) cube complex $\univ{X}$. If
\[
Y=Y_1\cup_H Y_2
\] 
for some hyperplane $H$ of $\univ{X}$ then 
\[
\Wh_X(Y)\cong \Wh_X(Y_1)\#_H\Wh_X(Y_2)\,.
\]
\end{proposition}
\begin{proof}
For notational convenience, write $W_i=\Wh_X(Y_i)$ for $i=1,2$, and $W=\Wh_X(Y)$. Our first task is to define the isomorphism $\Lk_{W_1}(H)\cong \Lk_{W_2}(H)$ needed to define the connected sum. 

A $(k-1)$-simplex $\sigma_K$ of $W_i$ corresponds to a bounding codimension-$k$ hyperplane $K$ of $Y_i$. This simplex $\sigma_K$ is contained in $\Lk_{W_i}(H)$ exactly when $K$ crosses $H$. A codimension-$k$ hyperplane that crosses $H$ bounds $Y_1$ if and only it bounds $Y_2$, so $\Lk_{W_1}(H)$ is naturally identified with $\Lk_{W_2}(H)$ as required. Let $L$ denote the natural copy of $\Lk_{W_1}(H)\equiv \Lk_{W_2}(H)$ in $W_1\#_HW_2$.

To see that $W\cong W_1\#_H W_2$, note that the $(k-1$)-simplices of $W_1\#_H W_2$ correspond bijectively to bounding codimension-$k$ hyperplanes of $Y$. Indeed, the simplices of $W_1\#_H W_2$ are of three kinds:
\begin{enumerate}[(i)]
\item simplices $\sigma_K$ in $L$ correspond to bounding codimension-$k$ hyperplanes of $Y$ that cross $H$;
\item simplices $\sigma_K$ of $W_1$ that are not adjacent to $H$ correspond to bounding codimension-$k$ hyperplanes of $Y_1$ disjoint from $H$;
\item simplices $\sigma_K$ of $W_2$ that are not adjacent to $H$ correspond to bounding codimension-$k$ hyperplanes of $Y_2$ disjoint from $H$.
\end{enumerate}
Since this list also covers all the codimension-$k$ hyperplanes of $Y$, and inclusions are preserved, the result follows.
\end{proof}

\subsection{Cut-set lemmas in connected sums}\label{sec: Splicing cut-set lemmas}

To understand splittings of fundamental groups of non-positively curved cube complexes, the key idea is to combine Proposition \ref{prop: Constructing Whitehead complexes by splicing} with an understanding of how the connectivity properties of simplicial complexes, in the sense of Definition \ref{def: Cut sets and connectivity}, change under connected sum. The first lemma is prototypical. Let $b_0(X)$ denote the 0th Betti number of a space $X$, i.e.\ the number of connected components.

\begin{lemma}\label{lem: Connected sum and components}
Suppose that $A$ and $B$ are simplicial complexes. Then
\[
b_0(A\#_v B)= b_0(B)
\]
unless $v$ is a cut vertex of $A$.
\end{lemma}
\begin{proof}
We work with reduced homology. Let $L$ denote the link of $v$ and $C(L)$ the cone on $L$. Then $B=B_v\cup_L C(L)$ so, because $C(L)$ is contractible, the Mayer--Vietoris theorem gives that
\[
\tilde{H}_0(L)\to \tilde{H}_0(B_v)\to \tilde{H}_0(B)\to 0
\]
is exact. The connected sum can be written as $A\#_v B= A_v\cup_L B_v$, so the Mayer--Vietoris theorem also gives that
\[
\tilde{H}_0(L)\to \tilde{H}_0(A_v)\oplus \tilde{H}_0(B_v)\to \tilde{H}_0(A\#_v B)\to 0
\]
is exact. Unless $v$ is a cut vertex in $A$, $\tilde{H}_0(A_v)$ is trivial, so the two sequences combine to show that
\[
 \tilde{H}_0(A\#_v B) \cong \operatorname{coker}(\tilde{H}_0(L)\to \tilde{H}_0(B_v) ) \cong \tilde{H}_0(B)
\]
as required.
\end{proof}

We will usually only apply the lemma when $B$ is connected, in which case it takes the following form.

\begin{lemma}\label{lem: Disconnected splicing}
Suppose that $A$ and $B$ are simplicial complexes. If $B$ is connected but $A\#_v B$ is disconnected, then $v$ is a cut vertex in $A$.
\end{lemma}

Since we will need to apply Lemma \ref{lem: Disconnected splicing} inductively, we also need to understand when cut vertices can arise, and more generally cut simplices.

\begin{lemma}\label{lem: Cut simplices from splicing}
Suppose that $A$ and $B$ are connected and a simplex $\sigma$ of $A$ is a cut simplex of $A\#_v B$. Then one of the following holds:
\begin{enumerate}[(i)]
\item $v$ is a cut vertex in $B$;
\item $\sigma$ is a cut simplex in either $A$ or $B$; or
\item $\sigma$ is contained in the link of $v$, and the simplex $\tau=\langle\sigma,v\rangle$ spanned by $\sigma$ and $v$ together is a cut simplex in both $A$ and $B$.
\end{enumerate}
In particular, if $A\#_v B$ has a cut simplex then either $A$ or $B$ has a cut simplex.
\end{lemma} 
\begin{proof}
Let $C=A\#_v B$ for notational convenience.  The hypothesis that $\sigma$ is a cut simplex of $C$ means that $C_\sigma$ is disconnected. The proof now divides into cases, according to whether $\sigma$ is in $A$ but not in $B$, or $\sigma$ is contained in the canonical copy of the link of $v$.

In the first case, when $\sigma$ is in $A$ but not in $B$, we have
\[
C_\sigma = A_\sigma\#_v B\,,
\]
so Lemma \ref{lem: Disconnected splicing} implies that $v$ is a cut vertex in $B$, as in item (i), or $A_\sigma$ is disconnected, i.e.\ $\sigma$ is a cut simplex in $A$.

In the second case, if $\sigma$ is in the link of $v$ then 
\[
C_\sigma = A_\sigma\#_v B_\sigma\,,
\]
so Lemma \ref{lem: Disconnected splicing} leads to two further subcases: either one of $A_\sigma$ or $B_\sigma$ is disconnected, or $v$ is a cut vertex in both $A_\sigma$ and $B_\sigma$. The former corresponds to item (ii), while the latter case gives item (iii).
\end{proof}

The above results will help us to understand free splittings of $\pi_1(X)$.  More generally, we will need to understand how arbitrary cut sets decompose.

\begin{lemma}\label{lem: Decomposing cut sets}
Suppose that $U$ is a full set of vertices of $A$ and $V$ is a full set of vertices of $B$, both non-adjacent to a common vertex $v$. If $U\cup V$ is a cut set of $A\#_v B$ then: either
\begin{enumerate}[(i)]
\item $V$ is a cut set in $B$; or
\item $U\cup \{v\}$ is a cut set in $A$.
\end{enumerate}
\end{lemma}
\begin{proof}
Again, let $C=A\#_v B$ for notational convenience. Because $v$ is adjacent to neither $U$ nor $V$,
\[
C_{U\cup V}= A_U\#_v B_V
\]
which is disconnected by hypothesis. Therefore, by Lemma \ref{lem: Disconnected splicing}, either $B_V$ is disconnected, giving item (i), or $v$ is a cut vertex in $A_U$, giving item (ii).
\end{proof}

\subsection{Cuts and the generalised Whitehead's lemma}\label{sec: Cuts and Whitehead's lemma}\

In this section, we combine the results developed so far to prove a sufficient condition for the fundamental group of a compact, non-positively curved cube complex $X$ to be one-ended. These conditions are phrased in terms of \emph{$k$-cuts}.

\begin{definition}[$k$-cuts]\label{def: k-cuts}
Consider a finite CAT(0) cube complex $Y$ equipped with an immersion $Y\to X$. If $\Wh_X(Y)$ has a finite cut set then $Y$ is said to be a \emph{cut}. More precisely, if $\Wh_X(Y)$ is $(k-1)$-vertex-connected but not $k$-vertex-connected -- i.e.\ if the smallest cut set for $\Wh_X(Y)$ has cardinality $k$ -- then $Y$ is said to be a \emph{$k$-cut}. If $\{H_1,\ldots,H_k\}$ is a minimal cut set for a $k$-cut $Y$ then we say that $Y$ is a cut \emph{between $H_1,\ldots,H_k$}.
\end{definition}

Our analysis of splittings of fundamental groups of non-positively curved cube complexes will rest on understanding cuts. We start with 0-cuts, which characterise free splittings. The results of \S\ref{sec: Splicing cut-set lemmas} explain how 0-cuts decompose.

\begin{lemma}\label{lem: Producing 0-cuts}
Consider a finite CAT(0) cube complex $Y$ equipped with a locally isometric immersion $Y\to X$, and let $Y=Y_1\cup_H Y_2$ for some hyperplane $H$ of $Y$. If $Y$ is a 0-cut then either one of $Y_1$ or $Y_2$ is a 0-cut, or both $Y_1$ and $Y_2$ are 1-cuts.
\end{lemma}
\begin{proof}
The immersion lifts to an embedding of $Y$ as a convex subcomplex of the universal cover $\univ{X}$, and the hyperplane $H$ of $Y$ then extends uniquely to a hyperplane of $\univ{X}$, still denoted by $H$. Proposition \ref{prop: Constructing Whitehead complexes by splicing} implies that 
\[
\Wh_X(Y)\cong \Wh_X(Y_1)\#_H \Wh_X(Y_2)\,,
\]
so the result follows from Lemma \ref{lem: Disconnected splicing}.
\end{proof}

To iterate the argument, we also need to analyse cut simplices.

\begin{lemma}\label{lem: Producing cut simplices}
Consider a finite CAT(0) cube complex equipped with a locally isometric immersion $Y\to X$, and let $Y=Y_1\cup_H Y_2$ for some hyperplane $H$ of $Y$. If $\Wh_X(Y)$ has a cut simplex, then one of $\Wh_X(Y_1)$ or $\Wh_X(Y_2)$ has a cut simplex.
\end{lemma}
\begin{proof}
This follows in the same way as Lemma \ref{lem: Producing 0-cuts}, using Proposition \ref{prop: Constructing Whitehead complexes by splicing} and Lemma \ref{lem: Cut simplices from splicing}.
\end{proof}

We now have all the tools needed to prove Theorem \ref{thm: Generalised Whitehead's lemma}, which we think of as a version of Whitehead's lemma for non-positively curved cube complexes.

\begin{proof}[Proof of Theorem \ref{thm: Generalised Whitehead's lemma}]
Let $Y$ be a CAT(0) cube complex, locally isometrically immersed in $X$. We prove that $\Wh_X(Y)$ is connected without cut simplices, by induction on the number of vertices of $Y$.

In the base case, $Y$ is a single vertex $y$, so $\Wh_X(Y)\cong\Lk_X(y)$ by Example \ref{eg: Whitehead graphs of vertices are links} and there is nothing to prove. If $Y$ has more than one vertex, we may choose a (necessarily separating) hyperplane $H$ of $Y$ and write $Y=Y_1\cup_H Y_2$. By the inductive hypothesis, both $\Wh_X(Y_1)$ and $\Wh_X(Y_2)$ are connected without cut simplices, and therefore so is $\Wh_X(Y)$ by Lemma \ref{lem: Producing cut simplices}.

In conclusion, $X$ has no 0-cuts, so Proposition \ref{prop: Detecting free splittings using Whitehead graphs} implies that $\pi_1(X)$ has at most one end.
\end{proof}

Theorem \ref{thm: Generalised Whitehead's lemma} provides a convenient certificate that $\pi_1(X)$ does not split freely.  However, the sufficient condition that it provides is not necessary.

\begin{example}\label{eg: Stupid example}
Let $X_0$ be the standard 2-torus, seen as a non-positively curved square complex with a single vertex $x_0$. Let $X=X_0\vee_{x_0} I$, where $I$ is a unit interval.  The link $\Lk_X(x_0)$ is the disjoint union of a 4-cycle and a point, and in particular is disconnected. However,  $\univ{X}$ is clearly one-ended.
\end{example}

Examples like \ref{eg: Stupid example} can be handled by collapsing free edges and free faces. (Free faces manifest themselves as vertices of valence 1 in the link of a vertex.) However, there are also non-positively curved square complexes $X$, with $\univ{X}$ one-ended, in which the link of a vertex is connected without valence-1 vertices, but contains a cut vertex.

The following example is based on \cite[Example 2.1]{kim_polygonal_2012}.

\begin{example}\label{eg: One-ended complex with cut vertices in the link}
Consider the rank-two free group $F=\langle a,b\rangle$, thought of as the fundamental group of the wedge of two circles $\Gamma=S^1\vee S^1$. Consider the element $w=abab^2ab^3$ of $F$ and realise the corresponding cyclic word by an immersion of a circle $\gamma:S^1\to \Gamma$. Let $X$ be the non-positively curved square complex obtained by gluing either end of a cylinder $S^1\times [0,1]$ to two copies of $\Gamma$, using $\gamma$ as the attaching map at each end. The links of the two vertices of $X$ are both isomorphic to the first subdivision of the Whitehead graph of $w$, which is connected, without valence-1 vertices, but has a cut vertex.

The fundamental group of $X$ is the double
\[
\pi_1(X)=F*_{\langle w\rangle} F
\]
by the Seifert--van Kampen theorem. Applying the automorphism $\phi$ of $F_2$ setting $\phi(a)=ab^{-2}$ and $\phi(b)=b$ sends $w\mapsto w'=ab^{-1}a^2b$.  We may realise $w'$ by an immersion $\gamma':S^1\to \Gamma$ and glue along a cylinder as before to obtain a new square complex $X'$ with $\pi_1(X')\cong\pi_1(X)$. However, the links of the two vertices of $X'$ are isomorphic to the first subdivision of the Whitehead graph of $w'$, which consists of two triangles glued along an edge, and in particular is connected with no cut vertices. Therefore, by Theorem \ref{thm: Generalised Whitehead's lemma}, $\pi(X)\cong\pi_1(X')$ is one-ended, even though the links of $X$ have cut vertices.
\end{example}

In the two-dimensional case, a square complex $X$ can always be modified so that the sufficient condition of Theorem \ref{thm: Generalised Whitehead's lemma} can be applied. A map of square complexes $\phi:X'\to X$ is called an \emph{unfolding of $X$} if (i) the restriction of $\phi$ to 1-skeleta is a combinatorial homotopy equivalence of graphs; and (ii) $\phi$ is a homeomorphism on the complements of the 1-skeleta. 

\begin{proposition}\label{prop: Grushko decomposition of a square complex}
Let $X$ be a compact, non-positively curved square complex. There is an unfolding $\phi:X'\to X$ such that
\[
X'=\Gamma\vee X_1\vee \ldots \vee X_m\vee S_1\vee\ldots \vee S_n
\]
is also a compact, non-positively curved square complex and, furthermore:
\begin{enumerate}[(i)]
\item $\Gamma$ is a graph;
\item each $X_i$ has the property that every link is connected without cut vertices or cut edges;
\item each $S_j$ is a square.
\end{enumerate}
\end{proposition}
\begin{proof}
The existence of $X'$ is \cite[Lemma 4.6]{wilton_rational_2024} stated for square complexes (see also \cite[\S3]{louder_uniform_2024}). The idea of the proof is an algorithmic unfolding procedure: one simply unfolds an edge of the 1-skeleton whenever the link of a vertex contains a cut vertex.

It remains to prove that $X'$ is non-positively curved. To do this, it suffices to check Gromov's link condition, that any cycle $\gamma$ in a link of a vertex of $X'$ should have length at least $2\pi$ in the angular metric \cite[Lemma II.5.6]{bridson_metric_1999}. Because $\phi$ is a homeomorphism on the complement of the 1-skeleton, the induced maps on links are injective on edges and hence immersions. Thus, $\gamma$ is also an immersed cycle in a link of $X$, so has length at least $2\pi$ by the link condition for $X$.
\end{proof}

The square complexes $X_i$ provided by the proposition satisfy the hypotheses of Theorem \ref{thm: Generalised Whitehead's lemma}. Thus, the wedge decomposition $X'$ actually gives the Grushko decomposition of $X$:
\[
\pi_1(X)\cong \pi_1(\Gamma)*\pi_1(X_1)*\ldots *\pi_1(X_m)\,,
\]
and each $\pi_1(X_i)$ is non-trivial and freely indecomposable.  Since the proof of Proposition \ref{prop: Grushko decomposition of a square complex} is algorithmic, Theorem \ref{thm: Generalised Whitehead's lemma} and the proposition together provide an effective, geometric method to determine the Grushko decomposition of the fundamental group of a non-positively curved square complex.

\subsection{Shepherd's theorem and \texorpdfstring{$1$}{1}-cuts}

The results of \S\ref{sec: Cuts and Whitehead's lemma} show that 0-cuts detect when an essential cube complex has more than one end, but they also illustrate the fact that the existence of 1-cuts makes it difficult to determine whether or not $\pi_1(X)$ splits. Therefore, for one-ended $\pi_1(X)$, we would like to be able to choose $X$ to have no 1-cuts.  In the 2-dimensional case, this can be done by \emph{unfolding} the 1-skeleton as in Proposition \ref{prop: Grushko decomposition of a square complex}. To handle the general case,  however, we can use a recent deep theorem of Shepherd.

A \emph{quarter-space} of a CAT(0) cube complex is the intersection of two transverse half-spaces. A quarter-space is said to be \emph{deep} if it is not contained in a bounded neighbourhood of the intersection of its two defining hyperplanes. The next result is  \cite[Theorem 1.5]{shepherd_semistability_2023}.

\begin{theorem}[Shepherd's theorem]\label{thm: Shepherd's theorem}
If $G$ is one-ended and the fundamental group of a non-positively curved cube complex, then $G=\pi_1(X)$ for an essential non-positively curved cube complex $X$, such that the universal cover $\univ{X}$ has the property that all half-spaces of $\univ{X}$ are one-ended and all quarter-spaces are deep.
\end{theorem}

The following lemma connects Shepherd's theorem to 1-cuts.

\begin{lemma}\label{lem: Half spaces and 1-cuts}
Let $\univ{X}$ be a one-ended, essential CAT(0) cube complex in which all quarter-spaces are deep. If $\univ{X}$ has a 1-cut, then some half-space of $\univ{X}$ has more than one end.
\end{lemma}
\begin{proof}
Let $Y\subseteq \univ{X}$ be a 1-cut, i.e.\ $Y$ is compact and convex and $\Wh_{\univ{X}}(Y)$ has a cut vertex $H$. Let $U$ be the half-space of $H$ that contains $Y$. 

We claim that $U$ is freely essential. Indeed, suppose that some hyperplane $K$ of $U$ had a compact half-space $V$. If $K$ were disjoint from $H$ then $K$ would also be a compact hyperplane in $\univ{X}$, contradicting the hypothesis that $\univ{X}$ is essential and one-ended. If $K$ crossed $H$ then $V$ would be a compact quarter-space in $\univ{X}$, contradicting the hypothesis that all quarter-spaces are deep.

Now, $\Wh_U(Y)$ is obtained from $\Wh_{\univ{X}}(Y)$ by deleting the star of $H$, so is disconnected. Therefore,  $U$ has more than one end by Proposition \ref{prop: Detecting free splittings using Whitehead graphs}, as required.
\end{proof}

Combining Shepherd's theorem with the lemma, we may always assume that one-ended cube complexes have no 1-cuts.

\begin{corollary}\label{cor: No 1-cuts}
If $G$ is the fundamental group of a compact, non-positively curved cube complex and $G$ is one-ended, then $G=\pi_1(X)$, where $X$ is a compact, essential, non-positively curved cube complex without 0-cuts or 1-cuts.
\end{corollary}

\section{Cyclic splittings}\label{sec: Cyclic splittings}

Having understood the free splittings of $\pi_1(X)$, we may henceforth assume that $\pi_1(X)$ is one-ended. Our next task is to adapt the analysis of the previous section to understand splittings of $\pi_1(X)$ over $\mathbb{Z}$. We restrict attention to the case where $\pi_1(X)$ is hyperbolic and, for brevity, we will say \emph{hyperbolic cube complex} to mean a compact, connected, non-positively curved cube complex $X$ with $\univ{X}$ (and hence $\pi_1(X)$) Gromov-hyperbolic. Let $\delta$ be a hyperbolicity constant for $\univ{X}$.

Cashen--Macura studied line patterns in free groups, i.e.\ relatively hyperbolic structures on free groups with cyclic parabolics \cite{cashen_line_2011}. They explained how to recognise cut pairs in Otal's decomposition space -- i.e.\ the Bowditch boundary -- by analysing cut pairs in generalised Whitehead graphs. Cashen went on to use this analysis to recognise relative cyclic splittings \cite{cashen_splitting_2016}.  Barrett used similar ideas to describe an algorithm that computes the Bowditch JSJ decomposition of any hyperbolic group \cite{barrett_computing_2018}.

Our first task is to adapt their ideas to the setting of hyperbolic cube complexes.

\subsection{Periodic \texorpdfstring{$2$}{2}-cuts}

Suppose that $\pi_1(X)$ is one-ended. If every splitting of $\pi_1(X)$ with cyclic edge groups is trivial, then $\pi_1(X)$ is said to be \emph{cyclically indecomposable}. The main theorem of this section is a combinatorial characterisation of cyclic indecomposibility. Just as we understood free splittings of $\pi_1(X)$ in terms of 0-cuts and 1-cuts of $X$, so cyclic splittings can be understood using 2-cuts.

As we will see, cyclic splittings give rise to 2-cuts, but the converse is not quite true: in order to construct a cyclic splitting from a 2-cut, we also need to consider the combinatorial types of the bounding hyperplane components.

\begin{definition}[Abstract hyperplane component]\label{def: Abstract hyperplane component}
Let $X$ be a compact, non-positively curved cube complex. An \emph{abstract hyperplane component} over $X$ consists of the following data:
\begin{enumerate}[(i)]
\item a compact, convex subcomplex $K$ of a hyperplane $H\subseteq \univ{X}$;
\item a non-trivial class $c_K$ in the reduced 0-cohomology $\tilde{H}^0(\Wh_H(K);\Z/2\Z)$.
\end{enumerate}
Two abstract hyperplane components $K\subseteq H$ and $K'\subseteq H'$ are of \emph{isomorphic type} if there is a deck transformation $\phi\in\pi_1(X)$ such that $\phi(K)=K'$, inducing an isomorphism  $\Wh_H(K)\cong \Wh_{H'}(K')$ under which $c_{K'}$ pulls back to $c_K$.

An \emph{orientation} of an abstract hyperplane component consists of the following extra datum:
\begin{enumerate}[(i)]
\setcounter{enumi}{2}
\item  a choice of one of the two half spaces of the hyperplane $H$.
\end{enumerate}
Two oriented abstract hyperplane components $K\subseteq H$ and $K'\subseteq H'$ are of \emph{opposite type} if they are of isomorphic type and, furthermore, the isomorphism $\phi$ can be chosen to reverse the chosen orientations.
\end{definition}

The reader may find it useful to think of the class $c_K$ as a non-trivial partition of the connected components of $\Wh_H(K)$.

This definition enables us to define a \emph{periodic} 2-cut.

\begin{definition}[Periodic 2-cut]\label{def: Periodic 2-cut}
Let $X$ be a compact, non-positively curved cube complex without 0- or 1-cuts. Consider a 2-cut $Y$ in $X$ between disjoint hyperplanes $H,H'$. Fix a non-trivial class $c_Y\in\tilde{H}^0(\Wh_X(Y)_{\{H,H'\}};\Z/2\Z)$. Define an oriented abstract hyperplane component corresponding to $H$ as follows:
\begin{enumerate}[(i)]
\item $K=H_Y$;
\item\label{item: Reduced cohomology condition} set $c_K$ to be the pullback of $c_Y$ under the inclusion of $\Wh_H(K)\cong\Lk_{\Wh_X(Y)}(H)$ into $\Wh_X(Y)_{\{H,H'\}}$;
\item the orientation is determined by the half-space of $H$ that contains $Y$.
\end{enumerate}
Define an oriented abstract hyperplane component $K'$ corresponding to $H'$ in the same way. The 2-cut $Y$ is said to be \emph{periodic} if the two oriented abstract hyperplane components are of opposite type, for some choice of $c_Y$. In this case, the element $\phi\in\pi_1(X)$ sending $H\to H'$ is called the \emph{associated covering transformation of $Y$}.
\end{definition}

\begin{remark}
Since $X$ has no 1-cuts, $\Wh_X(Y)_{\{H'\}}$ is connected. The Mayer--Vietoris sequence applied to the decomposition
\[
\Wh_X(Y)_{\{H'\}}=\Wh_X(Y)_{\{H,H'\}}\cup_{\Lk(H)}C(\Lk(H))
\]
then implies that the natural map
\[
\tilde{H}^0(\Wh_X(Y)_{\{H,H'\}};\Z/2\Z)\to \tilde{H}^0(\Wh_H(K);\Z/2\Z)
\]
is injective, so $c_K$ is indeed non-trivial, as required.
\end{remark}

We shall see that periodic 2-cuts characterise cyclic splittings. However, first we need to relate cyclic splittings to Whitehead complexes.

\begin{lemma}\label{lem: Hulls and splittings}
Let $X$ be an essential hyperbolic cube complex with $\pi_1(X)$ one-ended. Then $\pi_1(X)$ splits over $\Z$ if and only if some infinite cyclic subgroup acts cocompactly on some isometrically embedded subcomplex $\univ{Z}$ of $\univ{X}$ with $\Wh_X(\univ{Z})$ disconnected.
\end{lemma}
\begin{proof}
By the algebraic annulus theorem of Dunwoody--Swenson \cite{dunwoody_algebraic_2000}, and noting that $\pi_1(X)$ is torsion-free, $\pi_1(X)$ splits over a $\Z$ subgroup if and only if some subgroup $C\cong\Z$ satisfies that the number of relative ends $e(G,C)>1$. Because $\pi_1(X)$ is hyperbolic, $C$ is quasiconvex.

Choosing a base vertex $*$ in $\univ{X}$, the condition that $e(G,C)>1$ implies that, for some $R$, the $R$-neighbourhood of the orbit $C*$ separates $\univ{X}$ into at least two infinite components. Let $\univ{Z}$ be the cubical convex hull of the $R$-neighbourhood of $C*$, which still separates $\univ{X}$ into at least two infinite components. By Haglund's theorem (Theorem \ref{thm: Haglund's theorem}), $C$ acts cocompactly on $\univ{Z}$. Because $\univ{Z}$ separates $\univ{X}$ into two components, $\Wh_X(\univ{Z})$ is disconnected by Lemma \ref{lem: Topological type of Whitehead complex}, as required.

Conversely, suppose that some $C\cong\Z$ acts cocompactly on some convex subcomplex $\univ{Z}$ with  $\Wh_X(\univ{Z})$ disconnected, so $\univ{Z}$  separates $\univ{X}$ into two components by Lemma \ref{lem: Topological type of Whitehead complex}. Either some complementary component remains within a bounded distance of $\univ{Z}$, in which case some bounding hyperplane of $\univ{Z}$ is inessential, which is a contradiction, or all complementary components are unbounded away from $\univ{Z}$, meaning that $e(G,C)>1$, as required.
\end{proof}

In the setting of Lemma \ref{lem: Hulls and splittings}, it will be useful to know that the number of path components of the, \emph{a priori} infinite, complex $\Wh_X(\univ{Z})$ is finite.

\begin{lemma}\label{lem: Finitely many path components}
Let $X$ be a compact, non-positively curved cube complex without 0- or 1-cuts. Suppose that some infinite cyclic subgroup $\langle\phi\rangle$ of $\pi_1(X)$ acts cocompactly on some isometrically embedded subcomplex $\univ{Z}$  of $\univ{X}$. Then the number of path components of $\Wh_X(\univ{Z})$ is finite.
\end{lemma}
\begin{proof}
The compact quotient $Z=\langle\phi\rangle\backslash\univ{Z}$ has fundamental group $\Z\cong\langle\phi\rangle$, and therefore $H^1(Z;\Z)\cong\Z$. Since the first cohomology is generated by cocycles dual to hyperplanes, some hyperplane $K$ of $Z$ is crossed by any loop representing the generator $\phi$ of $\Z$. Since $\pi_1(K)$ is contained in the kernel of the dual map
\[
\Z\cong\pi_1(Z)\to\Z\,,
\]
it follows that $K$ is simply connected, so $K$ lifts to a hyperplane (still denoted by $K$) of $\univ{Z}$. By \cite[Lemma 2.3]{caprace_rank_2011}, after replacing $\phi$ by a further proper power we may assume that $\phi(K)$ is disjoint from $K$.

The hyperplane $K$ of $\univ{Z}$ extends to some hyperplane $H$ of $\univ{X}$. Cutting $\univ{Z}$ along the $\langle\phi\rangle$-orbit of $H$ decomposes it as an infinite gluing:
\[
\univ{Z}=\cdots \cup_{\phi^{-1}(H)} \phi^{-1}(Y)\cup_H Y\cup_{\phi(H)} \phi(Y)\cup_{\phi^2(H)}\cdots
\]
where $Y$ is the compact component resulting from cutting $\univ{Z}$ along $H$ and $\phi(H)$. It follows from this that the Whitehead complex $\Wh_X(\univ{Z})$ can be written as an ascending union
\[
\Wh_X(\univ{Z})=\bigcup_{n\geq 0}\Wh_X(Z_n)_{\{\phi^{-n}(H),\phi^n(H)\}}
\]
where 
\[
Z_n= \phi^{-n}(Y)\cup_{\phi^{1-n}(H)}\cdots \cup_{\phi^{n-1}(H)}\phi^{n-1}(Y)\,.
\]

In particular,
\[
\Wh_X(Z_n)= \Wh_X(Y)\#\Wh_X(Z_{n-1})\#\Wh_X(Y)
\]
by Proposition \ref{prop: Constructing Whitehead complexes by splicing}. Since $X$ has no 0- or 1-cuts, $\Wh_X(Y)$ is 1-vertex-connected, so
\[
b_0(\Wh_X(Z_n))=b_0(\Wh_X(Z_{n-1}))=b_0(\Wh_X(Z_0))
\]
by Lemma \ref{lem: Connected sum and components} and induction.

Recall that $K=H\cap \univ{Z}=H_Y$. Let $L=\Wh_H(K)\cong\Lk_{\Wh_X(Y)}(H)$ (by Lemma \ref{lem: Links in Whitehead complexes}), which is itself compact. Mayer--Vietoris applied to the decomposition
\[
\Wh_X(Z_n)=C(L)\cup_L \Wh_X(Z_n)_{\{\phi^{-n}(H),\phi^{n}(H)\}}\cup_L C(L)
\]
(where, as usual, $C(L)$ denotes the cone on $L$) gives the exact sequence
\[
H_0(L)\oplus H_0(L)\to H_0(\Wh_X(Z_n)_{\{\phi^{-n}(H),\phi^{n}(H)\}})\to H_0(\Wh_X(Z_n))\to 0\,.
\]
Thus, the number of connected components is bounded independently of $n$:
\begin{eqnarray*}
b_0(\Wh_X(Z_n)_{\{\phi^{-n}(H),\phi^{n}(H)\}})&\leq& b_0(\Wh_X(Z_n))+2b_0(L)\\
&= &b_0(\Wh_X(Z_0))+2b_0(L)\,.
\end{eqnarray*}
From this, it follows that $\Wh_X(\univ{Z})$ has finitely many connected components, as required.
\end{proof}

It follows from Lemmas \ref{lem: Hulls and splittings} and \ref{lem: Finitely many path components} that periodic 2-cuts characterise cyclic splittings.

\begin{theorem}[Cyclic splittings and periodic 2-cuts]\label{thm: Periodic slices and cyclic splittings}
Let $X$ be an essential hyperbolic cube complex. Suppose that $\pi_1(X)$ is one-ended, and suppose furthermore that $X$ has no 0- or 1-cuts.  Then $\pi_1(X)$ is cyclically decomposable if and only if there is a periodic 2-cut $Y\to X$.
\end{theorem}
\begin{proof}
By Lemma \ref{lem: Hulls and splittings}, $\pi_1(X)$ is cyclically decomposable if and only if there is a convex $\mathbb{Z}$-periodic subcomplex $\univ{Z}\subseteq\univ{X}$ that separates $\univ{X}$. By Lemma \ref{lem: Topological type of Whitehead complex}, $\univ{Z}$ separates $\univ{X}$ if and only if $\Wh_X(\univ{Z})$ is disconnected.

Suppose that there is a periodic 2-cut $Y\subseteq \univ{X}$ between hyperplanes $H,H'$. By definition, $H,H'$ define a cut pair in $\Wh_X(Y)$, and the corresponding oriented abstract hyperplane components have opposite type. Therefore, we may define the amalgamated subcomplex
\[
Y\cup_{\phi(H)} \phi(Y)\,.
\]
Iterating using all powers of $\phi$, we may set $\univ{Z}$ to be the infinite amalgamation
\[
\cdots \cup_{\phi^{-1}(H)} \phi^{-1}(Y)\cup_H Y\cup_{\phi(H)} \phi(Y)\cup_{\phi^2(H)}\cdots\,,
\]
which is $\langle\phi\rangle$-invariant and cocompact. By Proposition \ref{prop: Constructing Whitehead complexes by splicing}, $\Wh_X(\univ{Z})$ is constructed by iteratively taking connect sums of copies of $\Wh_X(Y)$ along copies of $H$, meaning that $\Wh_X(\univ{Z})$ is the infinite union
\[
\cdots \cup_L \Wh_X(Y)_{\{H,H'\}}\cup_L \Wh_X(Y)_{\{H,H'\}}\cup_L \Wh_X(Y)_{\{H,H'\}}\cup_L\cdots\,,
\]
where $L=\Lk_{\Wh_X(Y)}(H)\cong \Wh_H(H_Y)$ by Lemma \ref{lem: Links in Whitehead complexes}. The Mayer--Vietoris sequence for reduced cohomology (with coefficients in $\Z/2\Z$) asserts that
\[
0\to\tilde{H}^0(\Wh_X(\univ{Z}))\to \prod_{n\in\Z} \tilde{H}^0(\Wh_X(Y)_{\{H,H'\}}) \to \prod_{n\in\Z} \tilde{H}^0(L)
\]
is exact. The right-hand map sends the non-trivial element
\[
(c_Y)_{n\in\Z}\in\prod_{n\in\Z} \tilde{H}^0(\Wh_X(Y)_{\{H,H'\}})
\]
to $(c_K)-(\phi^*(c_{K'}))=0$, so $\tilde{H}^0(\Wh_X(\univ{Z}))$ is non-trivial, and $\Wh_X(\univ{Z})$ is indeed disconnected as required.

Conversely, suppose that there is a convex $\Z$-periodic subcomplex $\univ{Z}\subseteq\univ{X}$ with $\Wh_X(\univ{Z})$ disconnected, with the infinite cyclic subgroup $\Z$ generated by some covering transformation $\phi$. By Lemma \ref{lem: Finitely many path components}, $\tilde{H}^0(\Wh_X(\univ{Z});\Z/2\Z)$ is finite so, after replacing $\phi$ by a power, we may assume that it acts trivially on $\tilde{H}^0(\Wh_X(\univ{Z});\Z/2\Z)$. As in the proof of Lemma \ref{lem: Finitely many path components}, $\univ{Z}$ has a compact hyperplane $K$ that separates the two ends, and extends to a hyperplane $H$ of $\univ{X}$. After replacing $\phi$ by a proper power, we may assume that $\phi(H)$ is disjoint from $H$.

Let $\univ{Z}_\pm$ be the two half-spaces of $H$, so
\[
\univ{Z}=\univ{Z}_-\cup_H\univ{Z}_+\,.
\]
We claim that $\Wh_X(\univ{Z}_+)$ and $\Wh_X(\univ{Z}_-)$ are connected. Indeed, let $Y$ be the compact component of $\univ{Z}$ that results from cutting along $H$ and $\phi(H)$. Then $\univ{Z}_+$ may be written as the ascending union of compact, convex subcomplexes $Y_n$ where
\[
Y_n=Y\cup_{\phi(H)} \phi(Y)\cup_{\phi^2(H)}\ldots \cup_{\phi^{n-1}(H)}\phi^{n-1}(Y)
\]
and, correspondingly, the Whitehead complex of $\univ{Z}_+$ may be written as an ascending union 
\[
\Wh_X(\univ{Z}_+)=\bigcup_{n\geq 1}\Wh_X(Y_n)_{\phi^n(H)}\,.
\]
Thus, if $\Wh_X(\univ{Z}_+)$ were disconnected then it would follow that $\Wh_X(Y_n)_{\phi^n(H)}$ was disconnected for some $n$. If so then $Y_n$ would be a 1-cut in $X$, contradicting the assumption. Thus, $\Wh_X(\univ{Z}_+)$ is connected and, symmetrically, so is $\Wh_X(\univ{Z}_-)$.  Since
\[
\Wh_X(\univ{Z})=\Wh_X(\univ{Z}_-)\#_H\Wh_X(\univ{Z}_+)
\]
is disconnected,  it follows that $H$ is a cut vertex in both $\Wh_X(\univ{Z}_-)$ and $\Wh_X(\univ{Z}_+)$ by Lemma \ref{lem: Disconnected splicing}.

Now, $H'=\phi(H)$ is a hyperplane of $\univ{Z}_+$ disjoint from $H$, and cutting along $H'$ decomposes $\univ{Z}_+$ as
\[
\univ{Z}_+=Y\cup_{H'}\phi(\univ{Z}_+)
\]
where $Y=\univ{Z}_+\cap \phi(\univ{Z}_-)$; note that $H$ is a bounding hyperplane of $Y$. Hence,
\[
\Wh_{X}(\univ{Z}_+)=\Wh_X(Y)\#_{H'}\Wh_{X}(\phi(\univ{Z}_+))
\]
and so, since we saw above that $\Wh_{X}(\phi(\univ{Z}_+))\cong\Wh_X(\univ{Z}_+)$ is connected, $Y$ is a 2-cut between $H$ and $H'$  by Lemma \ref{lem: Decomposing cut sets}. 

Finally, note that $Y$ is a periodic 2-cut. Indeed, setting $H'=\phi(H)$, it only remains to construct a suitable $c_Y\in\tilde{H}^0(\Wh_X(Y)_{\{H,H'\}})$ that satisfies Definition \ref{def: Periodic 2-cut}. To do this, as above we consider the exact sequence
\[
0\to\tilde{H}^0(\Wh_X(\univ{Z}))\stackrel{\alpha}{\to} \prod_{n\in\Z} \tilde{H}^0(\Wh_X(Y)_{\{H,H'\}}) \stackrel{\beta}{\to} \prod_{n\in\Z} \tilde{H}^0(L)
\]
provided by Mayer--Vietoris, where $L=\Wh_H(K)$. Note that  $\Z\cong\langle\phi^*\rangle$ acts on all terms in the exact sequence, and the maps $\alpha$ and $\beta$ are $\Z$-equivariant.

Since $\Wh_X(\univ{Z})$ is disconnected, there is some non-zero  $a \in \tilde{H}^0(\Wh_X(\univ{Z}))$, and $\alpha(a)=(b_n)_{n\in\Z}$ is also non-zero because $\alpha$ is injective. But $\phi^*$ acts trivially on $\tilde{H}^0(\Wh_X(\univ{Z}))$, so $\phi^*(a)=a$, and $(b_n)_{n\in\Z}$ is actually a non-zero constant sequence $(b)$. Set $c_Y=b$. 

Write $i$ for the inclusion of $L$ into $\Wh_X(Y)_{\{H,H'\}}$ and $j$ for the inclusion of $L$ into $\Wh_X(\phi^{-1}(Y))_{\{\phi^{-1}(H),H\}}$. Then
\[
\beta(b)=(i^*(b)-\phi^*\circ j^*(b))\,,
\]
so $i^*(b)=\phi^*(j^*(b))$ by exactness. Therefore, setting $c_K=i^*(b)$ and $c_{K'}=j^*(b)$, we have $c_K=\phi^*(c_{K'})$ as required.
\end{proof}

\subsection{A bound on the widths of \texorpdfstring{$2$}{2}-cuts}

Although Theorem \ref{thm: Periodic slices and cyclic splittings} provides a useful characterisation of the existence of cyclic splittings, in order to make it effective we need an \emph{a priori} bound on the complexity of the 2-cuts that arise.  Such a bound is the main result of this section, and is proved in Theorem \ref{thm: Slice length bound} below. It is phrased in terms of the \emph{width} of a 2-cut.

\begin{definition}[Width of a $k$-cut]\label{def: Width of a cut}
Let $X$ be a compact, non-positively curved cube complex, and consider a $k$-cut $Y\to X$ between bounding hyperplanes $\{H_1,\ldots, H_k\}$ for $k\geq 2$. The \emph{width} of $Y$ is the shortest $\ell^2$-distance between any distinct $H_i$ and $H_j$.
\end{definition}

The next lemma provides some geometric control on cuts, assuming that there are no 1-cuts.

\begin{lemma}\label{lem: Cuts are contained in hulls}
Let $\univ{X}$ be a one-ended a CAT(0) cube complex without 0- or 1-cuts. Consider a $k$-cut $Y$ between bounding hyperplanes $H_1,\ldots, H_k$. Then $Y\cap \mathrm{Hull}\left(\bigcup_{i=1}^k H_i\right)$ is also a $k$-cut between $H_1,\ldots, H_k$.
\end{lemma}
\begin{proof}
Let $K_1,\ldots K_l$ be the hyperplanes of $Y$ that have a half-space $U_j$ such that $H_i\subseteq U_j$ for all $i$. Write $Y_n=Y\cap \bigcap_{j=1}^n U_j$ for each $0\leq n\leq l$. Note that
\[
Y_l=Y\cap \mathrm{Hull}\left(\bigcup_{i=1}^k H_i\right)\,.
\]
To prove the lemma, we show by induction on $n$ that each $Y_n$ is a $k$-cut.

In the base case, $Y_0=Y$ and there is nothing to prove. For the inductive step, let $n\geq 1$. If $K_n$ is disjoint from $Y_{n-1}$ then $Y_{n-1}=Y_n$ and there is, again, nothing to prove. Otherwise,  $K_n$ decomposes $Y_{n-1}$ as
\[
Y_{n-1}=Y_n\cup_{K_n} Y'_n
\]
where $Y_n=Y_{n-1}\cap U_n$.  Then $H_1,\ldots, H_k$ all bound $Y_n$ and not $Y'_n$ (since they are disjoint from $K_n$, and $K_n$ cannot separate $Y$ from any $H_i$). By Proposition \ref{prop: Constructing Whitehead complexes by splicing}
\[
\Wh_X(Y_{n-1})\cong\Wh_X(Y_n)\#_{K_n}\Wh_X(Y'_n)\,,
\]
and $K_n$ is not a cut vertex in $\Wh_X(Y'_n)$ because $\univ{X}$ has no 1-cuts. Therefore $\{H_1,\ldots, H_k\}$ is a cut set in $\Wh_X(Y_n)$ by Lemma \ref{lem: Decomposing cut sets}, so $Y_n$ is also a $k$-cut between $H_1,\ldots, H_k$. Taking $n=l$, the result follows.
\end{proof}

We now turn our attention to some CAT(0) geometry of hyperplanes. Consider a 2-cut $Y$ between hyperplanes $H_1$ and $H_2$, and let $\gamma$ be an $\ell^2$-geodesic between $H_1$ and $H_2$. In order to prove the desired bound, we need to find long chains of mutually disjoint hyperplanes that cross $\gamma$ at an angle that isn't too acute. A hyperplane $K$ is said to be \emph{$\epsilon$-almost orthogonal} to $\gamma$ if $\gamma$ crosses $K$ at an angle of at least $\epsilon$. 

The next lemma shows that many of the hyperplanes crossed by $\gamma$ will be almost orthogonal, for a suitable value of $\epsilon$. The proof uses elementary Euclidean trigonometry.

\begin{lemma}\label{lem: Many almost orthogonal hyperplanes}
Let $d\geq 2$, and let $\gamma$ be an $\ell^2$-geodesic in a $d$-dimensional CAT(0) cube complex $\univ{X}$. Suppose that $\gamma$ crosses two hyperplanes $K_1$ and $K_2$ consecutively. Then at least one of $K_1$ and $K_2$ is $\epsilon_d$-almost orthogonal to $\gamma$, where $\epsilon_d=\tan^{-1}\left(\frac{1}{2\sqrt{d-1}}\right)$.
\end{lemma}
\begin{proof}
If $\gamma$ crosses $K_1$ and $K_2$ simultaneously then the result is immediate, because $K_1$ and $K_2$ are orthogonal to each other and $\epsilon_d<\pi/4$.

Otherwise, suppose that $\gamma(0)\in K_1$ and $\gamma$ crosses $K_1$ at an angle $\theta<\epsilon_d$. Because the (closed) regular neighbourhood $N(K_1)$ of $K_1$ is isometric to $K_1\times [-1/2,1/2]$, we have $\gamma(t)\in N(K_1)$ whenever $|t|\sin\theta\leq 1/2$. In particular, if $\gamma(t_0)$ is in the boundary of $N(K_1)$ then $|t_0|=1/(2\sin\theta)$. Let $p$ be the orthogonal projection of $\gamma(t_0)$ to $K_1$. The three points $\gamma(0)$, $\gamma(t_0)$, $p$ together form a right-angled Euclidean triangle, so
\[
d(\gamma(0),p)=|t_0|\cos\theta=\frac{1}{2}\cot\theta\,.
\]
Because $\theta<\epsilon_d$, it follows that
\[
d(\gamma(0),p)> \frac{1}{2}\cot\epsilon_d=\sqrt {d-1}
\]
which implies that $\gamma(0)$ and $p$ are separated by a hyperplane in the $(d-1)$-dimensional cube complex $K_1$. Since every such hyperplane extends uniquely to a hyperplane in $\univ{X}$ orthogonal to $K_1$, it follows that $\gamma$ crosses a hyperplane in $N(K_1)$ orthogonal to $K_1$.

The first such hyperplane is by assumption $K_2$. Let $\gamma(t_1)\in K_2$ and let $q$ be the orthogonal projection of $\gamma(0)$ to $K_2$. Then the three points $\gamma(0)$, $\gamma(t_1)$, $q$ together span a right-angled Euclidean triangle, whence $\gamma$ makes an angle
\[
\pi/2-\theta >  \pi-\epsilon_d>\pi/4>\epsilon_d
\]
with $K_2$, as required.
\end{proof}

The next lemma guarantees that a sufficiently distant pair of hyperplanes will be separated by many mutually disjoint hyperplanes, all themselves almost orthogonal to a geodesic.

\begin{lemma}\label{lem: Pencil between distant hyperplanes}
Let $H_1$ and $H_2$ be hyperplanes in a $d$-dimensional CAT(0) cube complex $\univ{X}$, and let $\gamma$ be a shortest $\ell^2$-geodesic between them. For every $n$ there is a $C$ such that, if the length of $\gamma$ is greater than $C$, then there exists a sequence of mutually disjoint hyperplanes
\[
K_1,\ldots, K_n
\]
that each separate $H_1$ from $H_2$, all of which are $\epsilon_d$-almost orthogonal to $\gamma$.
\end{lemma}
\begin{proof}
Let $N$ be the Ramsey number $R(n,d+1)$. Recall that the $\ell^1$ distance between $H_1$ and $H_2$ is equal to the number of hyperplanes crossed by $\gamma$. Since the $\ell^1$ and $\ell^2$ metrics are quasi-isometric, there is a $C$ such that, if $\gamma$ has $\ell^2$ length at least $C$, then it crosses at least $2N$ hyperplanes. By Lemma \ref{lem: Many almost orthogonal hyperplanes}, at least $N$ of these hyperplanes are $\epsilon_d$-almost orthogonal to $\gamma$.  Since no set of $d+1$ hyperplanes of $\univ{X}$ pairwise intersect, it follows from the definition of $N$ that there is a pairwise disjoint set
\[
\{K_1,\ldots, K_n\}
\]
of $n$ hyperplanes that intersect $\gamma$ $\epsilon_d$-almost orthogonally.

Finally, note that any hyperplane $K$ that crosses $H_i$ meets it orthogonally, and therefore if $\gamma$ also crossed $K$ there would be a compact geodesic triangle in $\univ{X}$ with two right angles, contradicting the CAT(0) condition.  Therefore, each $K_j$ is disjoint from $H_1$ and $H_2$, and therefore separates them from each other.
\end{proof}

We needed to find almost orthogonal hyperplanes because of the following lemma, which will enable us to bound the diameter of a hyperplane intersecting a suitably chosen 2-cut.

\begin{lemma}\label{lem: Sine rule estimate}
Suppose that an $\ell^2$ geodesic $\gamma$ intersects a hyperplane $H$ at an angle $\theta$. Let $x$ be the point of intersection and let $y$ be a point in $H$. Then
\[
d(x,y)\sin\theta \leq d(y,\gamma)\,.
\]
In particular, if $H$ is $\epsilon$-almost orthogonal to $\gamma$, $d(x,y)\leq d(y,\gamma)/\sin\epsilon$.
\end{lemma}
\begin{proof}
Let $z$ be the closest point on $\gamma$ to $y$, and consider the geodesic triangle with vertices $x$, $y$ and $z$. Let $\bar{\theta}$ be the angle in the Euclidean comparison triangle at the vertex corresponding to $x$, which is at least $\theta$ by \cite[Proposition 1.7(4)]{bridson_metric_1999}. Now, the Euclidean sine rule gives
\[
d(y,z)\geq d(x,y)\sin\bar{\theta}\geq d(x,y)\sin\theta
\]
as required.
\end{proof}

The proof of the claimed bound on the widths of 2-cuts uses the pigeonhole principle, applied to isomorphism types of abstract hyperplane components of a given diameter. 

\begin{remark}\label{rem: Finitely many abstract bounding hyperplane components of given diameter}
Let $X$ be a compact, non-positively curved cube complex, and let $R\geq 0$. The set of isomorphism types of abstract bounding hyperplane components $Z$ over $X$ of diameter at most $R$ is finite. Indeed, the bound on diameter implies that there are only finitely many possible subcomplexes $Z$ of hyperplanes $H$ of $\univ{X}$, up to the action of $\pi_1(X)$; since $\Wh_H(Z)$ is compact for each $Z$, the number of possible cohomology classes in $\tilde{H}^0(\Wh_H(Z);\Z/2\Z)$ is also finite.
\end{remark}

This finiteness is the basis for our argument using the pigeonhole principle.

\begin{theorem}\label{thm: Slice length bound}
Let $X$ be an essential, $d$-dimensional, hyperbolic cube complex without 0-cuts or 1-cuts. Then $\pi_1(X)$ is cyclically indecomposable if and only if there is a bound on the widths of 2-cuts in $X$.
\end{theorem}
\begin{proof}
The `only if' direction is similar to the proof of Theorem \ref{thm: Periodic slices and cyclic splittings}. If $\pi_1(X)$ has a cyclic splitting then, by Lemma \ref{lem: Hulls and splittings}, there is an infinite cyclic subgroup $\langle \phi\rangle$ acting cocompactly on some isometrically embedded subcomplex $\univ{Z}\subseteq \univ{X}$ with $\Wh_X(\univ{Z})$ disconnected. As in the proof of Theorem \ref{thm: Periodic slices and cyclic splittings}, we can find a  hyperplane $H$ of $\univ{Z}$ such that, for any $n$, cutting $\univ{Z}$ along $H$ and $\phi^n(H)$ gives a 2-cut between $H$ and $\phi^n(H)$. Since $d(H,\phi^n(H))\to\infty$ with $n$, it follows that there is no bound on the widths of 2-cuts, as required.

For the converse, let $Y$ be a 2-cut between hyperplanes $H_1$ and $H_2$. By Lemma \ref{lem: Cuts are contained in hulls}, we may replace $Y$ by its intersection with $\Hull(H_1\cup H_2)$, and so assume that $Y$ is contained in $\Hull(H_1\cup H_2)$. Let $\gamma$ be an $\ell^2$-geodesic between $H_1$ and $H_2$.  It is easy to see that $H_1\cup \gamma\cup H_2$ is $2\delta$-quasiconvex in the $\ell^2$ metric and therefore, by Theorem \ref{thm: Haglund's theorem}, $\Hull(H_1\cup H_2)$ is contained in the $R$-neighbourhood of $H_1\cup\gamma\cup H_2$, for some uniform $R$.

Let $M$ be the number of isomorphism types of oriented abstract hyperplane components over $X$ of diameter at most $2R\sqrt {4d-3}$ (which is finite by Remark \ref{rem: Finitely many abstract bounding hyperplane components of given diameter}). Let $M'$ be such that any geodesic crossing $M'$ hyperplanes has an $\ell^2$-length of at least $(1+\sqrt{4d-3})R$, and let $N=M+2M'$. By Lemma \ref{lem: Pencil between distant hyperplanes}, if the width of $Y$ is greater than a constant $C$, then we can find a sequence of $M+1$ hyperplanes
\[
K_0,K_1,\ldots, K_M\,,
\]
all separating $H_1$ from $H_2$, all $\epsilon_d$-almost orthogonal to $\gamma$, and such that each $K_i$ intersects $\gamma$ at distance greater than $(1+\sqrt{4d-3})R$ from both $H_1$ and $H_2$.

From each $K_i$ we define an oriented abstract hyperplane component $Z_i$ as follows. After setting $Z_i=K_i\cap Y$, we need to define a non-trivial cohomology class $c_{Z_i}\in\tilde{H}^0(\Wh_{K_i}(Z_i))$ where, as usual, we consider cohomology with coefficients in $\Z/2\Z$. Cutting $Y$ along $K_i$ decomposes it as
\[
Y=Y_1\cup_{K_i} Y_2\,,
\]
and Proposition \ref{prop: Constructing Whitehead complexes by splicing} then tells us that
\[
\Wh_X(Y)=\Wh_X(Y_1)\#_{K_i} \Wh_X(Y_2)\,.
\]
In particular, $\Wh_X(Y)_{\{H_1,H_2\}}$ contains a copy of $L_i=\Wh_{K_i}(Z_i)$, which is canonically isomorphic to the links of $K_i$ in both $\Wh_X(Y_1)$ and $\Wh_X(Y_2)$ by Lemma \ref{lem: Links in Whitehead complexes}. Since $Y$ is a 2-cut, there is a non-trivial class $c_Y$ in the reduced cohomology group $\tilde{H}^0(\Wh_X(Y)_{\{H_1,H_2\}})$. The cohomology class $c_{Z_i}$ is then defined as the pullback of $c_Y$ along the inclusion of $L_i$ into $\Wh_X(Y)_{\{H_1,H_2\}}$. We claim that $c_{Z_i}\neq 0$.

Indeed, Mayer--Vietoris applied to the decomposition
\[
\Wh_X(Y)_{\{H_1,H_2\}}=\Wh_X(Y_1)_{\{H_1,K_i\}}\cup_{L_i}\Wh_X(Y_2)_{\{K_i,H_2\}}
\]
shows that the pullback $b$ of $c_Y$ to $\tilde{H}^0(\Wh_X(Y_1)_{\{H_1,K_i\}})$ (without loss of generality) is non-zero. But $c_{Z_i}$ is the pullback of $b$ to $\tilde{H}^0(L_i)$, so if $c_{Z_i}=0$ then, by Mayer--Vietoris applied to the decomposition
\[
\Wh_X(Y_1)_{H_1}=\Wh_X(Y_1)_{\{H_1,K_i\}}\cup_{L_i} C(L_i)\,,
\]
we see that $b$ is the image of a non-trivial cohomology class in $\tilde{H}^0(\Wh_X(Y_1)_{H_1})$, contradicting the hypothesis that $X$ has no 1-cuts. Hence, $c_{Z_i}\neq 0$ as claimed.

We next prove a bound on the diameter of any $Z_i$. Indeed, let $x$ be the point at which $K_i$ intersects $\gamma$.  Since $Y$ is contained in the closed $R$-neighbourhood of $H_1\cup\gamma\cup H_2$, $Z_i$ is the union of two closed subsets:
\[
Z_i = P\cup Q
\]
where $P=Z_i\cap( N_R(H_1)\cup N_R(H_2))$ and $Q=Z_i\cap N_R(\gamma)$.

By Lemma \ref{lem: Sine rule estimate}, the distance from any $y\in Q$ to $x$ is bounded above by
\[
\frac{R}{\sin\epsilon_d}= \frac{R}{\sin\cot^{-1}(2\sqrt{d-1})} = R\sqrt{4d-3}\,.
\]
We next argue that $P$ is empty. Indeed, if not then there is $y\in P\cap Q$, since $Z_i$ is connected and $P$ and $Q$ are both closed. But such a $y$ is within distance $R$ of $H_1\cup H_2$ and within $R\sqrt{4d-3}$ of $x$, so $K_i$ intersects $\gamma$ at distance less than $R(1+\sqrt{4d-3})$ from either $H_1$ or $H_2$, contradicting the construction of the $K_i$ above. Thus, $P$ is indeed empty, so $Z_i=Q$ and $\diam(Z_i)=\diam(Q) \leq 2R\sqrt{4d-3}$. 

Therefore, by the pigeonhole principle and the choice of $M$, there exist distinct $i<j$ such that $Z_i$ and $Z_j$ are isomorphic oriented abstract hyperplane components. Cutting $Y$ along $K_i$  and $K_j$ decomposes it as
\[
Y=Y_1\cup_{K_i} Y'\cup_{K_j} Y_2\,.
\]
Applying Lemma \ref{lem: Decomposing cut sets} twice, $Y'$ is a 2-cut between $K_i$ and $K_j$, while $Z_i$ and $Z_j$ are the corresponding bounding hyperplane components. 

Finally, note that the inclusions of $L_i=\Wh_{K_i}(Z_i)$ and $L_j=\Wh_{K_j}(Z_j)$ into $\Wh_X(Y)_{\{H_1,H_2\}}$ factor through the inclusion of $\Wh_{X}(Y')_{\{K_i,K_j\}}$. Therefore, setting $c_{Y'}$ to be the pullback of $c_Y$ to $\Wh_{X}(Y')_{\{K_i,K_j\}}$, we see that $c_{Y'}$ pulls back to both $c_{Z_i}$ and $c_{Z_j}$. This completes the proof that $Y'$ is a periodic 2-cut.
\end{proof}

\subsection{From \texorpdfstring{$2$}{2}-cuts to \texorpdfstring{$k$}{k}-cuts}

In the previous section, we proved a uniform bound on the widths of 2-cuts. In this section, we will upgrade this to a uniform bound on the widths of $k$-cuts, where the bound is independent of $k$.  The key result is Lemma \ref{lem: Quasitree lemma}, which encapsulates the fact that a finite collection of hyperplanes coarsely has the structure of a quasi-tree. I am grateful to Mark Hagen for pointing out that this lemma should be true, and that it would simplify the proof of my argument.

The proof of Lemma  \ref{lem: Quasitree lemma} is inspired by the projection-complex machinery of Bestvina--Bromberg--Fujiwara \cite{bestvina_constructing_2015}\footnote{Anthony Genevois has communicated to me a shorter, direct proof of the lemma that avoids the Bestvina--Bromberg--Fujiwara machinery \cite{genevois_personal_2024}.}. Consider a collection $\V=\{V_i\}$ of convex subcomplexes of a hyperbolic, CAT(0) cube complex $\univ{X}$. (In  $\cite{bestvina_constructing_2015}$ this collection would typically be infinite. For us it will be finite, but of unbounded cardinality.) For each $U\in\V$, we consider the closest-point projection $\pi_U:\univ{X}\to U$. For each $U$, Bestvina--Bromberg--Fujiwara use this to define a \emph{distance function} on pairs of elements of $\V\smallsetminus\{U\}$ via
\[
d^\pi_U(V,W):=\diam(\pi_U(V)\cup\pi_U(W))\,.
\]
They go on to state \emph{projection complex axioms} for the distance functions $d^\pi_\bullet$, with respect to a fixed constant $\theta>0$; see \cite[\S1.3]{bestvina_acylindrical_2019} for a concise statement of the axioms.  The next proposition asserts that, as long as the elements of $\V$ are pairwise sufficiently distant from each other, the projection complex axioms are satisfied.

\begin{proposition}\label{prop: Projection complex axioms}
Let $\univ{X}$ be a $\delta$-hyperbolic CAT(0) cube complex and let $\V$ be a finite set of convex subcomplexes of $\univ{X}$.
There are constants $C$ and $\theta>0$, depending only on $\univ{X}$, such that, as long as the elements of $\V$ are pairwise at distance at least $C$, the following hold for all  distinct $U,V,W,Y\in\V$:
\begin{enumerate}[(i)]
\item $d^\pi_U(V,W)=d^\pi_U(W,V)$;
\item $d^\pi_U(V,W)+d^\pi_U(W,Y)\geq d^\pi_U(V,Y)$;
\item $d^\pi_U(V,V)\leq\theta$;
\item if $d^\pi_U(V,W)>\theta$ then $d^\pi_V(U,W)\leq\theta$;
\item $\#\{Z\in\V \mid d^\pi_Z(V,W)>\theta\}<\infty$
\end{enumerate}
\end{proposition}
\begin{proof}

Properties (i) and (ii) are obvious, while property (v) is immediate by hypothesis since $\V$ is assumed to be finite. Property (iii) is known as the \emph{contracting property} for quasi-convex subspaces of $\delta$-hyperbolic spaces, and a standard exercise shows that it holds for any $\theta\geq 8\delta$ as long as $C\geq 2\delta$. Finally, property (iv) is known as the \emph{Behrstock inequality}. A sketch is given in \cite[p.\ 133]{sisto_what_2019}.
\end{proof}

In order to construct a quasitree, it is necessary to adjust the distance functions $d^\pi_\bullet$ to strengthen the Behrstock inequality. For this, we follow Bestvina--Bromberg--Fujiwara--Sisto \cite{bestvina_acylindrical_2019}. From $d^\pi_\bullet$, they construct a new family of distance functions $d_\bullet$ with improved properties.

\begin{proposition}\label{prop: Improved projection complex axioms}
For $d^\pi_\bullet$ and $\theta$ as above, there are functions $d_\bullet$ such that 
\[
|d^\pi_\bullet-d_\bullet|\leq 2\theta\,,
\]
and that satisfy the following axioms for all distinct $U, V,W,Y\in\V$:
\begin{enumerate}[(i)]
\item $d_U(V,W)=d_U(W,V)$;
\item $d_U(V,W)+d_U(W,Y)\geq d_U(V,Y)$;
\item $d_U(V,V)\leq 11\theta$;
\item if $d_U(V,W)> 11\theta$ then $d_W(V,Y)=d_W(U,Y)$ unless $Y=W$;
\item $\#\{Z\in\V \mid d_Z(V,W)>\theta\}<\infty$.
\end{enumerate}
\end{proposition}
\begin{proof}
This is \cite[Proposition 4.14]{bestvina_acylindrical_2019}.
\end{proof}

\begin{remark}\label{rem: Strengthened Behrstock inequality}
Proposition \ref{prop: Improved projection complex axioms}(iv) should be thought of as a strengthened version of the Behrstock inequality. Indeed, taking $Y=V$ in Proposition \ref{prop: Improved projection complex axioms}(iv) shows that, if $d_U(V,W)>11\theta$, then $d_W(U,V)=d_W(V,V)\leq 11\theta$ by Proposition \ref{prop: Improved projection complex axioms}(iii).
\end{remark}

Bestvina--Bromberg--Fujiwara--Sisto use the strengthened projection complex axioms to construct an infinite quasitree $\P_K(\V)$. In our situation, $\P_K(\V)$ will be finite, and we will use it to find $U\in\V$ onto which the diameter of the projection of all the other elements of $\V$ is bounded.

\begin{definition}[Projection complex]\label{def: Projection complex}
For any $K>0$, the \emph{projection complex} $\P_{K}(\V)$ is a graph with vertex set $\V$. Distinct vertices $U,V$ are joined by an edge if $d_W(U,V)\leq K$ for all $W\neq U,V$.
\end{definition}

For large enough $K$, they go on to define a \emph{standard path} in $\P_{K}(\V)$ between any $U,V\in\V$; we will let $[U,V]$ denote this standard path.  Similarly, we will let $(U,V)$ denote $[U,V]\smallsetminus\{U,V\}$, etc. (In \cite{bestvina_acylindrical_2019}, the set of vertices in $(U,V)$ is denoted by $\mathbf{Y}_K(U,V)$.)  We summarise the properties of standard paths in the following proposition.

\begin{proposition}[Properties of standard paths]\label{prop: Properties of standard paths}
Suppose $K\geq 33\theta$. Let $U,V\in \V$. The \emph{standard path} $[U,V]$ is an embedded interval in $\P_K(\V)$ between $U$ and $V$. Furthermore, the following hold.
\begin{enumerate}[(i)]
\item \label{item: Path compatibility} If $U',V'\in [U,V]$ then $[U',V']\subseteq [U,V]$.
\item \label{item: Standard paths are quasi-geodesics} If $n$ is the length of the standard path $[U,V]$ then the shortest path in $\P_{K}(\V)$ from $U$ to $V$ has length at least $\lfloor\frac{n}{2}\rfloor+1$.
\item \label{item: Projections depend on paths} If $W\in (U,V]$ then $d_U(V,W)\leq K$.
\end{enumerate}
\end{proposition}
\begin{proof}
The proofs of \cite[\S2--3]{bestvina_acylindrical_2019} use a parameter $\theta$ that satisfies the \emph{strong projection axioms}. In our context, the strong projection axioms are provided by Proposition \ref{prop: Improved projection complex axioms}, so we must replace $\theta$ by $11\theta$.

The vertices of $(U,V)$ consist of those $W\in\V$ such that $d_W(U,V)>K$. Setting $W_1<W_2$ if $d_{W_1}(U,W_2)>K$ defines a total order on the vertices of $(U,V)$ by \cite[Proposition 2.3]{bestvina_acylindrical_2019}, to which we may add $U$ as the least element and $V$ as the greatest element; \cite[Lemma 3.1]{bestvina_acylindrical_2019} asserts that these are indeed the vertices of an embedded path.

Item (\ref{item: Path compatibility}) now follows immediately from \cite[Corollary 2.5]{bestvina_acylindrical_2019} and item (\ref{item: Standard paths are quasi-geodesics}) is \cite[Lemma 3.7]{bestvina_acylindrical_2019}. For item (\ref{item: Projections depend on paths}), there is nothing to prove if $W=V$; otherwise $d_W(U,V)>K>11\theta$ so $d_U(V,W)\leq 11\theta < K$ by Remark \ref{rem: Strengthened Behrstock inequality}.
\end{proof}

Bestvina--Bromberg--Fujiwara--Sisto also provide a very precise description of triangles of standard paths.

\begin{lemma}[Triangles of standard paths]\label{lem: Triangles of standard paths}
Suppose $K\geq 33\theta$. Let $U,V,W\in\V$. The complement
\[
[U,V]\smallsetminus ([U,W]\cup[W,V])
\]
in $\P_{K}(\V)$ is connected, and contains at most two vertices.
\end{lemma}
\begin{proof}
This is \cite[Lemma 3.6]{bestvina_acylindrical_2019}.
\end{proof}

With these facts in hand, we can prove our desired bound on diameters of projections.

\begin{lemma}\label{lem: Uniformly bounded projections}
Let $\univ{X}$ be a $\delta$-hyperbolic CAT(0) cube complex and let $\V$ be a finite set of unbounded convex subcomplexes of $\univ{X}$ with at least two elements. There is a constant $C$, depending only on $\univ{X}$, such that, as long as the elements of $\V$ are pairwise at distance at least $C$, there is $U\in \mathbb{Y}$ such that 
\[
\diam\left(\bigcup_{V\in\V\smallsetminus U}\pi_U(V)\right) \leq 400\theta\,.
\]
\end{lemma}
\begin{proof}
Let $C$ and $\theta$ be as in Proposition \ref{prop: Projection complex axioms}, and consider the projection complex $\P_{33\theta}(\V)$, defined using the distance functions $d_\bullet$ from Proposition \ref{prop: Improved projection complex axioms}.  Since $\V$ is finite, there are finitely many standard paths; fix $[U,V]$ of maximal length. Let $V'$ be the vertex of $[U,V]$ adjacent to $U$, and let $W'$ be any other vertex of $\P_{33\theta}(\V)$ adjacent to $U$. 

First, note that $U$ is not contained in the standard path from $W'$ to $V$. Indeed, if $U$ were contained in $[W',V]$ then, by Proposition \ref{prop: Properties of standard paths}(\ref{item: Path compatibility}), $[U,V]$ would be a proper sub-path of $[W',V]$, contradicting the choice of $U$.  Note also that the standard path $[U,W']$ is a single edge, by Proposition \ref{prop: Properties of standard paths}(\ref{item: Standard paths are quasi-geodesics}).

Next, consider the triangle of standard paths $[U,W']\cup [W',V]\cup [V,U]$. Combining the facts from the previous paragraph with Lemma \ref{lem: Triangles of standard paths}, we see that $V'$ is joined to $W'$ by an edge-path of length at most 5, avoiding $U$. By the definition of $\P_{33\theta}(\V)$ and induction on the length of this edge-path, it follows that $d_U(V',W')\leq 5\times 33\theta=165\theta$. 

Now consider an arbitrary $W\neq U$ in $\V$ and let $W'$ be the vertex of $[W,U]$ adjacent to $U$. Combining the estimate from the previous paragraph with Proposition \ref{prop: Properties of standard paths}(\ref{item: Projections depend on paths}) gives
\[
d_U(V',W)\leq d_U(V',W') + d_U(W',W)\leq 165\theta+33\theta= 198\theta\,.
\]
It follows from Proposition \ref{prop: Improved projection complex axioms} that $d^\pi_U(V',W)\leq 200\theta$, and so $d^\pi_U(W_1,W_2)\leq 400\theta$ for any $W_1,W_2\neq U$ in $\V$, by the triangle inequality. This is the required bound, by the definition of $d^\pi_\bullet$.
\end{proof}

The geometric lemma that we need can now be easily deduced.

\begin{lemma}\label{lem: Quasitree lemma}
Let $\univ{X}$ be a $d$-dimensional, $\delta$-hyperbolic CAT(0) cube complex. There is a constant $D$ with the following property. Suppose that $\V$ is any finite, non-empty set of unbounded convex subcomplexes of $\univ{X}$ and that $d(U,V)>D$ for every distinct $U,V\in\V$. Then, there exists a hyperplane $H$ and $U\in \V$ such that $H$ separates $U$ from every $V\in\V\smallsetminus\{U\}$.
 \end{lemma}
\begin{proof}
Let $C$ and $\theta$ be as in Lemma \ref{lem: Uniformly bounded projections}. By Theorem \ref{thm: Haglund's theorem}, there is $r$ such that, for any subset $Z\subseteq \univ{X}$ of diameter at most $400\theta$, the diameter of $\Hull(Z)$ is at most $r$.

For each $U\in\V$, let $\widehat{U}$ be the convex hull of the $(\sqrt{d}+r)$-neighbourhood of $U$, which is contained in the $R$-neighbourhood of $U$, for some uniform $R$, by Theorem \ref{thm: Haglund's theorem}. Let $\widehat{\V}:=\{\widehat{U}\mid U\in\V\}$ and let $D=C+2R$. 

As long as the elements of $\V$ are at distance at least $D$ from each other, it follows that the elements of $\widehat{\V}$ are at distance at least $C$ from each other. Hence, by Lemma \ref{lem: Uniformly bounded projections}, there is $\widehat{U}$ such that
\[
\diam\left(\bigcup_{\widehat{V}\in\widehat{\V}\smallsetminus \widehat{U}}\pi_{\widehat{U}}(\widehat{V})\right) \leq 400\theta\,.
\]
Let $W\subseteq\widehat{U}$ be the convex hull of $\bigcup_{\widehat{V}\in\widehat{\V}\smallsetminus \widehat{U}}\pi_{\widehat{U}}(\widehat{V})$, so $\diam W\leq r$.

Since $W$ contains a point on the frontier of $\widehat{U}$, it follows from the bound on the diameter of $W$ and the definition of $\widehat{U}$ that $d(U,W)\geq \sqrt{d}$. In particular, by Lemma \ref{lem: l2 geodesics and separating hyperplanes} there is a hyperplane $H$ that separates $U$ from $W$; note that $H$ intersects $\widehat{U}$ non-trivially by construction.

Finally, consider any $V\neq U$ in $\V$. It remains to prove that $H$ separates $U$ from $V$. Since $H$ separates $U$ from $W$, it suffices to prove that $H$ is disjoint from $V$ and does not separate $W$ from $V$.  Because hyperplanes are convex and $H$ intersects $\widehat{U}$, if $H$ also intersected $V$ then it would intersect $W$, a contradiction. So $H$ is disjoint from $V$.  If $H$ separated $W$ from $V$ then a shortest geodesic $\gamma$ from $V$ to $W$ would cross $H$ by Lemma \ref{lem: l2 geodesics and separating hyperplanes}; but $\gamma$ is also a shortest geodesic from $V$ to $\widehat{U}$, so $H$ would separate $V$ from $\widehat{U}$, contradicting the fact that $H$ intersects $\widehat{U}$.
\end{proof}

We can now upgrade our bound on the widths of 2-cuts to a bound on the widths of $k$-cuts.

\begin{theorem}\label{thm: Uniform bound on k-cuts}
Let $X$ be an essential, $d$-dimensional, hyperbolic cube complex without 0-cuts or 1-cuts. Then $\pi_1(X)$ is cyclically indecomposable if and only if there is a uniform bound, independent of $k$, on the widths of all $k$-cuts with $k\geq 2$.
\end{theorem}
\begin{proof}
If there is a bound on the widths of $k$-cuts for all $k\geq 2$ then, in particular, there is a bound on the widths of all $2$-cuts. Therefore, $X$ is cyclically essential and $\pi_1(X)$ is cyclically indecomposable by Theorem \ref{thm: Slice length bound}.

For the converse implication, let $C$ be the bound from Theorem \ref{thm: Slice length bound}. By Theorem \ref{thm: Haglund's theorem}, there is $R$ such that the hull of the $C$-neighbourhood of any convex subset is contained in the $R$ neighbourhood of that subset. Let $D$ be the constant from Lemma \ref{lem: Quasitree lemma}.

Let $Y$ be a $k$-cut between a collection of $k\geq 2$ hyperplanes $\{H_1,\ldots H_k\}$, and suppose that the width of $Y$ is greater than $2R+D$.  Suppose, furthermore, that $k$ is minimal such that there exists a $k$-cut of width greater than $2R+D$.

For each $i$, let $U_i$ be the convex hull of the $C$-neighbourhood of $H_i$. Then the $U_i$ are pairwise at distance at least $D$. 

By Lemma \ref{lem: Quasitree lemma}, there is a hyperplane $K$ that, without loss of generality, separates $U_1$ from the remaining $H_i$. In particular, $K$ is at distance at least $C$ from $H_1$ and separates all the other $H_i$ from $H_1$. Thus, cutting $Y$ along $K$ gives
\[
Y=Y_1\cup_K Y_2
\]
where $H_1$ bounds $Y_1$ and the remaining $H_i$ all bound $Y_2$.

By the minimality of $k$, $\{H_2,\ldots, H_k\}$ is not a cut set in $\Wh_X(Y_2)$. Therefore, Lemma \ref{lem: Decomposing cut sets} implies that $Y_1$ is a 2-cut between $H_1$ and $K$, which is of width at least $C$ by construction. By the choice of $C$, it follows that  $\pi_1(X)$ is cyclically decomposable, as required.
\end{proof}

\section{One-ended subgroups}\label{sec: One-ended subgroups}

In this section, we prove both our main theorem about cubulated hyperbolic groups and its corollary about one-relator groups.

\subsection{Cubulated hyperbolic groups}

With the uniform bound on the widths of cuts in hand, it is easy to prove the main theorem in the cyclically indecomposable case.

\begin{theorem}\label{thm: Subgroups in the rigid case}
Let $X$ be a hyperbolic cube complex. If $\pi_1(X)$ is one-ended and cyclically indecomposable then it has a one-ended quasiconvex subgroup of infinite index.
\end{theorem}
\begin{proof}
By Corollary \ref{cor: No 1-cuts},  we may without loss of generality assume that $X$ is essential, without 0- or 1-cuts. Let $C$ be the uniform bound on widths of $k$-cuts given by Theorem \ref{thm: Uniform bound on k-cuts}.

Choose any hyperplane $H$ of $X$, and fix an elevation $\univ{H}$ to the universal cover $\univ{X}$. Because $\pi_1 H$ is separable in $G=\pi_1X$ by the work of Agol \cite{agol_virtual_2013} and Haglund--Wise \cite{haglund_special_2008}, there is a finite-index subgroup $G_0$ of $G$ such that
\[
d(\univ{H},g\univ{H})>C
\]
for any $g\in G_0\smallsetminus\Stab(\univ{H})$.

Let $X_0$ denote the quotient $G_0\backslash \univ{X}$ and let $H_0$ be the image of $\univ{H}$, which is a hyperplane of $X_0$. Let $X'$ be (a component of) the result of cutting $X_0$ along $H_0$. Since $X'$ is finite and maps locally isometrically to $X$, $\pi_1(X')$ is a quasiconvex subgroup of $G$. Furthermore, since $\pi_1(X')$ is a vertex-group of a non-trivial graph-of groups decomposition for $G_0$, it has infinite index in $G_0$, and hence in $G$. It remains to show that $\pi_1(X')$ is one-ended.

With a view to a contradiction, suppose that $\pi_1(X')$ is freely decomposable. By Proposition \ref{prop: Detecting free splittings using Whitehead graphs}, there is a 0-cut $Y\to X'$. But $\Wh_{X'}(Y)$ is obtained from $\Wh_X(Y)$ by deleting the $G_0$-translates of $\univ{H}$, so some finite set of translates
\[
g_1\univ{H},\ldots, g_k\univ{H}
\]
(with $g_i\in G_0$) defines a cut set in $\Wh_X(Y)$. Since $\Wh_X(Y)$ has no 0-cut or 1-cuts, $k\geq 2$. Therefore $Y$ is a $k$-cut in $\univ{X}$ of width greater than $C$, contradicting the choice of $C$.
\end{proof}

By passing to the Grushko and Bowditch JSJ decompositions, the main theorem follows.

\begin{proof}[Proof of Theorem \ref{thm: Main theorem for cubulated hyperbolic groups}]
Since $\pi_1(X)$ is a finitely generated, torsion-free group, it is freely indecomposable if and only if it is one-ended, by Stallings' theorem \cite{stallings_torsion-free_1968}. If $\pi_1(X)$ is freely decomposable, Grushko's theorem implies that either $\pi_1(X)$ is free or it has a one-ended free factor, which is necessarily a quasiconvex subgroup of infinite index. Therefore, we may assume that $\pi_1(X)$ is itself one-ended.

Consider the Bowditch JSJ decomposition of $\pi_1(X)$ \cite{bowditch_cut_1998}.  If the decomposition has only one vertex and no edges, then that vertex is either of surface or of rigid type. In the former case, $\pi_1(X)$ is a surface group. In the latter case, $\pi_1(X)$ is cyclically indecomposable, so it has a one-ended, quasiconvex subgroup of infinite index by Theorem \ref{thm: Subgroups in the rigid case}.

Now consider the case where the JSJ decomposition is non-trivial. Every vertex group $V$ is quasiconvex \cite[Proposition 1.2]{bowditch_cut_1998}, and in particular finitely generated and hyperbolic. If some vertex group $V$ is not itself free, then by Grushko's theorem it has a one-ended free factor $V'$, which provides a one-ended, quasiconvex subgroup, necessarily of infinite index in $\pi_1(X)$. On the other hand, if every vertex group is free, then $\pi_1(X)$ is either a surface group or contains a one-ended, quasiconvex subgroup by \cite[Theorem 3]{wilton_one-ended_2011}.
\end{proof}

\subsection{One-relator groups}

Combining the result for cubulated hyperbolic groups with recent joint work of the author with Louder \cite{louder_negative_2022,louder_uniform_2024}, and also work of Linton \cite{linton_one-relator_2025} and Gardam--Kielak--Logan \cite{gardam_surface_2023}, we can also answer Question \ref{qu: General subgroup question} for one-relator groups. The key is Puder's notion of \emph{primitivity rank} \cite{puder_primitive_2014}.

\begin{definition}[Primitivity rank]\label{def: Primitivity rank}
Let $w$ be a non-trivial element of a free group $F$. The \emph{primitivity rank} $\pi(w)$ is the minimal rank of a subgroup $w\in H\subseteq F$
 such that $w$ is not a primitive element of $H$.
 \end{definition}
 
By convention, $\min\varnothing =\infty$, so $\pi(w)=\infty$ if and only if $w$ is primitive in $F$ (and hence in every subgroup containing it). As a result of this convention,  $\pi(w)\geq 1$ for all $w$. (It can be convenient to extend the definition of $\pi$ to the trivial element $1\in F$ and set $\pi(1)=0$.)
 
It has recently been discovered that $\pi(w)$ determines several features of the one-relator group $G=F/\llangle w\rrangle$. These results can be summarised in the following theorem.

\begin{theorem}\label{thm: Structure of one-relator groups}
Let $F$ be free, $w$ a non-trivial element of $F$ and $G=F/\llangle w\rrangle$. Then:
\begin{enumerate}[(i)]
\item $\pi(w)=\infty$ if and only if $G$ is free;
\item $\pi(w)=1$ if and only if $G$ has torsion;
\item if $\pi(w)=2$ then $G$ has a two-generator, one-relator subgroup $P$ that is not free;
\item if $\pi(w)>2$ then $G$ is hyperbolic and cubulated.
\end{enumerate}
\end{theorem}
\begin{proof}
Item (i) is a result of the classical theorem of Whitehead that, for non-trivial $w$, $G$ is free if and only if $w$ is primitive \cite[Proposition 5.10]{lyndon_combinatorial_1977}. It is easy to see that $\pi(w)=1$ if and only if $w$ is a proper power in $F$; a classical theorem of Karrass--Magnus--Solitar asserts that this is true exactly when $G$ has torsion \cite[Proposition 5.17]{lyndon_combinatorial_1977}. Item (iii) is a theorem of Louder and the author \cite[Theorem 1.5]{louder_negative_2022}, while item (iv) is a theorem of Linton \cite[Theorem 8.2]{linton_one-relator_2025} (using deep work of Wise \cite{wise_structure_2021}, as well as further work of Louder and the author \cite{louder_uniform_2024}).
\end{proof}
 
\begin{remark}\label{rem: Newman and Wise}
The one-relator group $G$ is also hyperbolic and cubulated when $\pi(w)=1$, by work of Newman \cite{newman_some_1968} and Wise \cite{wise_structure_2021}. However, we will not need either of these facts here.
\end{remark}

The final piece of the jigsaw is a recent result of Gardam--Kielak--Logan, who handled the two-generator, one-relator case of Question \ref{qu: General subgroup question} \cite[Theorem 1.5]{gardam_surface_2023}.

\begin{theorem}[Gardam--Kielak--Logan]\label{thm: GKL theorem}
Let $G$ be a two-generator, one-relator group. If every subgroup of infinite index in $G$ is free, then $G$ is either free or the fundamental group of a closed surface.
\end{theorem}

Combining these results with Theorem \ref{thm: Main theorem for cubulated hyperbolic groups}, we obtain the desired result for all one-relator groups.

\begin{proof}[Proof of Theorem \ref{thm: Main theorem for one-relator groups}]
A one-relator group $G=F/\llangle w\rrangle$ is free if $\pi(w)=0$ or $\pi(w)=\infty$. Since $G$ is infinite and every subgroup of infinite index is free, $G$ is torsion-free, so $\pi(w)>1$. If $\pi(w)>2$ then $G$ is hyperbolic and cubulated by Theorem \ref{thm: Structure of one-relator groups}(iv), so the result follows from Theorem \ref{thm: Main theorem for cubulated hyperbolic groups}. Thus, it remains to handle the case when $\pi(w)=2$.

In this case, $G$ has a two-generator, one-relator subgroup $P$ that is not free by Theorem \ref{thm: Structure of one-relator groups}(iii). The index of $P$ in $G$ is finite by hypothesis, and every subgroup of infinite index in $P$ is free. By Theorem \ref{thm: GKL theorem}, $P$ is either $\Z^2$ or $\Z\rtimes\Z$, so the same is true of $G$ by Bieberbach's theorem.
\end{proof}

\section{Strebel's theorem}\label{sec: Strebel}

Let $G$ be an $n$-dimensional Poincar\'e duality group. Strebel proved that any subgroup of infinite index in $G$ must have strictly smaller cohomological dimension \cite{strebel_remark_1977}, paralleling an older theorem of Whitehead about manifolds \cite{whitehead_immersion_1961}.

\begin{theorem}[Strebel theorem]\label{thm: Strebel's theorem}
Let $G$ be a $PD_n$ group. If $H$ is a subgroup of infinite index in $G$, then $\mathrm{cd}(H)<n$.
\end{theorem}

By Stallings' theorem, Theorem \ref{thm: Main theorem for cubulated hyperbolic groups} can be thought of as a converse to Strebel's theorem when $n=1$ and $G$ is hyperbolic and cubulated. It is therefore natural to wonder whether, if $G$ remains hyperbolic and cubulated, the converse is also true for larger $n$.\footnote{This was stated as an open question in an earlier version of this paper.} The purpose of this section is to note that this converse fails.  I am grateful to Genevieve Walsh and Macarena Arenas for explaining this to me.

The key point is that Strebel's theorem is not optimal. Indeed, the following stronger result is given as an exercise in Brown's book \cite[Exercise 4, p.\ 203]{brown_cohomology_1994}. 

\begin{theorem}[Strong Strebel theorem]\label{thm: Strong Strebel's theorem}
Let $G$ be a group of type $FP$ and cohomological dimension $n$, and suppose $H^n(G,\Z G)$ is finitely generated as an abelian group. If $H$ is a subgroup of infinite index in $G$, then $\mathrm{cd}(H)<n$.
\end{theorem}

With the stronger theorem in hand, we can provide an example of a 3-dimensional group, which is not a $PD_3$ group, but such that every subgroup of infinite index is at most 2-dimensional.

\begin{example}\label{eg: No higher qc Strebel converse}
Let $L$ be a square-free flag triangulation of the 2-dimensional torus (which exists by a lemma of Dranishnikov \cite[Proposition 2.1]{dranishnikov_boundaries_1999}). Let $W_L$ be the corresponding right-angled Coxeter group, and let $G=[W_L,W_L]$ be the commutator subgroup. Then $G$ is a hyperbolic, cubulated group, acting properly and cocompactly on the Davis cube complex $\univ{X}$; let $X$ be the quotient cube complex. The link of every vertex of $\univ{X}$, and hence also of $X$, is isomorphic to the torus $L$.

The group $G$ has geometric and cohomological dimension 3, but is not a Poincar\'e duality group. Indeed, the Gromov boundary $\partial_\infty G$ is the Pontryagin sphere \cite[Remark 4.4(1)]{przytycki_flag-no-square_2009}, which is 2-dimensional with $\check{H}^2(\partial_\infty G;\Z)\cong\Z$ and $\check{H}^1(\partial_\infty G;\Z)\cong\Z^\infty$ .  By the Bestvina--Mess theorem \cite[Corollary 1.3]{bestvina_boundary_1991},
\[
H^*(G;\Z G)\cong \check{H}^{*-1}(\partial_\infty G;\Z)\,,
\]
so $G$ has cohomological dimension 3,  and every subgroup of infinite index has cohomological dimension at most 2 by Theorem \ref{thm: Strong Strebel's theorem}. But $G$ is not $PD_3$, because $H^2(G,\Z G)\cong \check{H}^1(\partial_\infty G;\Z)\neq 0$.
\end{example}

Thus, Theorem \ref{thm: Main theorem for cubulated hyperbolic groups} should actually be thought of as a conditional converse to Theorem \ref{thm: Strong Strebel's theorem} in dimension 2.

\section{Questions}\label{sec: Questions}

This final section contains some open questions that are suggested by the results of this paper.

As explained in \S\ref{sec: Strebel}, Theorem \ref{thm: Main theorem for cubulated hyperbolic groups} can be thought of as a converse to the strong Strebel theorem (Theorem \ref{thm: Strong Strebel's theorem}) in dimension 2, for cubulated hyperbolic groups. Therefore, we can ask whether similar converses hold in higher dimensions.

\begin{question}\label{qu: Higher strengthened Strebel converse}
Let $X$ be a hyperbolic cube complex with $G=\pi_1(X)$ and  $\cd(G)=n>2$. If the cohomological dimension of every subgroup of infinite index in $G$ is less than $n$, does it follow that $H^n(G,\Z G)$ is finitely generated as an abelian group?
\end{question}

The proof of Theorem \ref{thm: Main theorem for cubulated hyperbolic groups} made essential use of the separability of hyperplane stabilisers, which is closely related to the fact that $X$ is virtually special. Removing this condition would necessitate a different approach.

\begin{question}\label{qu: Main question in non-special case}
Let $X$ be a compact, non-positively curved cube complex with every subgroup of infinite index free. Must $\pi_1(X)$ be either a free group or a surface group?
\end{question}

\noindent Combined with Theorem \ref{thm: Main theorem for cubulated hyperbolic groups}, the well known flat-closing conjecture would imply a positive answer.

In a similar vein, removing the cubulation hypothesis from Theorem \ref{thm: Main theorem for cubulated hyperbolic groups} would be a major breakthrough. The following question reiterates \cite[Questions 1.7 and 1.11]{bestvina_questions_????}.

\begin{question}[Gromov, Whyte]\label{qu: Main question for hyperbolic groups}
Let $G$ be an infinite hyperbolic group. If every subgroup of infinite index in $G$ is free, must $G$ be either a free group or a surface group?
\end{question}
\noindent Of course, if every one-ended hyperbolic group had a surface subgroup, then Question \ref{qu: Main question for hyperbolic groups} would have an affirmative answer.

The next question addresses the slight mismatch between the conclusions of Theorems \ref{thm: Main theorem for cubulated hyperbolic groups} and \ref{thm: Main theorem for one-relator groups}.

\begin{question}\label{qu: Refined one-relator statement}
Let $G$ be a two-generator, one-relator group such that every finitely generated subgroup of infinite index is free. Must $G$ be isomorphic to either $\Z$ or a solvable Baumslag--Solitar group?
\end{question}

\noindent (Note that the two two-generator surface groups, $\Z^2$ and $\Z\rtimes\Z$, are $BS(1,1)$ and $BS(1,-1)$ respectively.)

As explained in \S\ref{sec: Cuts and Whitehead's lemma}, any non-positively curved \emph{square} complex can be `unfolded' until Theorem \ref{thm: Generalised Whitehead's lemma} applies, providing an effective geometric computation of the Grushko decomposition.  I would like to extend this method to higher-dimensional cube complexes.\footnote{Touikan's algorithm provides a general computation of the Grushko decomposition for any torsion-free group with solvable word problem \cite{touikan_detecting_2018}. However, a cubical method would be preferable.}

The first step is to recognise whether or not the cube complex $X$ is essential. Sam Shepherd has explained to me that being essential can be recognised algorithmically: the key observation is that an inessential CAT(0) cube complex contains a half-space that is minimal with respect to inclusion, and the bounding hyperplane of a minimal half-space can be recognised  \cite{shepherd_personal_2024}.

Given essential $X$, the next step is to determine whether or not $\pi_1(X)$ splits freely.   The following question asks whether or not the sufficient condition of Theorem \ref{thm: Generalised Whitehead's lemma} can always be applied.

\begin{question}\label{qu: Effective Whitehead}
Let $G$ be a one-ended cubulable group. Is $G\cong\pi_1(X)$, where $X$ is essential and no vertex of $X$ has a link with a separating simplex?
\end{question}

\noindent Shepherd's theorem (Theorem \ref{thm: Shepherd's theorem}) implies that there is an $X$ without separating vertices in the links, but does not provide information about higher-dimensional simplices.  Regardless of the answer to Question \ref{qu: Effective Whitehead}, it is also important to find an algorithmic condition to recognise whether or not $X$ admits 0- or 1-cuts.

Assuming an understanding of free splittings, one may pose similar questions about splittings over $\Z$.  It is not clear if there is an analogue of Theorem \ref{thm: Generalised Whitehead's lemma} for cyclic splittings.

\begin{question}\label{qu: Local cyclic splitting criterion}
Let $X$ be a compact, non-positively curved cube complex with $\pi_1(X)$ one-ended. Is there a condition on the links of vertices of $X$ that guarantees that $\pi_1(X)$ does not split over $\Z$?
\end{question}

\noindent For instance, it seems probable that the techniques of this article can be used to prove that, if $X$ is a hyperbolic square complex and the links of vertices of $X$ are connected without cut vertices or cut pairs, then $\pi_1(X)$ does not split over $\Z$. But it would be better still to provide a local criterion for higher-dimensional cube complexes, or outside the hyperbolic setting.

As a step towards understanding cyclic splittings outside the hyperbolic setting, the next question asks whether the hyperbolicity hypothesis can be removed from Theorem \ref{thm: Periodic slices and cyclic splittings}.

\begin{question}\label{qu: Periodic 2-cuts without hyperbolicity}
Let $X$ be a compact, essential, non-positively curved cube complex without 0- or 1-cuts. If $\pi_1(X)$ splits over $\Z$, must $X$ admit a periodic 2-cut?
\end{question}

\noindent The converse statement, that $\pi_1(X)$ splits over $\Z$ if $X$ admits a periodic 2-cut, follows from the algebraic annulus theorem of Dunwoody--Swenson \cite{dunwoody_algebraic_2000}.

\bibliographystyle{plain}

\Addresses

\end{document}